\documentclass[reqno,10pt]{amsart}

\usepackage[marginpar=2cm]{geometry}
\usepackage{hyperref}
\usepackage{amscd, amssymb , amsmath,epsfig,subfig}
\usepackage{fullpage, braket}
\usepackage{cancel}
\usepackage{mathrsfs}
\usepackage{epsfig}
\usepackage{color}
\definecolor{brightlavender}{rgb}{0.75, 0.58, 0.89}
\definecolor{carnelian}{rgb}{0.7, 0.11, 0.11}

\usepackage[all]{xy}
\usepackage{verbatim}
\usepackage{stmaryrd}

\newcommand{\Proj}{\operatorname{Proj}}

\newcommand{\ptr}{\operatorname{ptr}}

\newtheorem{Df}{Definition}[section]
\newtheorem{definition}[Df]{Definition}
\newtheorem{theorem}[Df]{Theorem}

\newtheorem{prop}[Df]{Proposition}
\newtheorem{proposition}[Df]{Proposition}
\newtheorem{lemma}[Df]{Lemma}

\newtheorem{remark}[Df]{Remark}

\newtheorem{corollary}[Df]{Corollary}
\newtheorem{conjecture}[Df]{Conjecture}

\newcommand{\cut}{{\operatorname{cut}}}
\newcommand{\rcoev}{\stackrel{\longrightarrow}{\operatorname{coev}}}
\newcommand{\rev}{\stackrel{\longrightarrow}{\operatorname{ev}}\!\!}
\newcommand{\lev}{\stackrel{\longleftarrow}{\operatorname{ev}}\!\!}
\newcommand{\lcoev}{\stackrel{\longleftarrow}{\operatorname{coev}}}
\newcommand{\p}[1]{\ensuremath{\bar {#1}}}
\newcommand{\Ar}{\wa{\operatorname{Ar}}}
\newcommand{\bp}[1]{{\left(#1\right)}}

\newcommand{\cat}{\mathcal{C}}

\newcommand{\catd}{\mathcal{D}}
\newcommand{\C}{\ensuremath{\mathbb{C}}}
\newcommand{\Z}{\ensuremath{\mathbb{Z}}}
\newcommand{\R}{\ensuremath{\mathbb{R}}}

\newcommand{\g}{\ensuremath{\mathfrak{g}}}

\newcommand{\End}{\operatorname{End}}
\newcommand{\Hom}{\operatorname{Hom}}

\newcommand{\unit}{\ensuremath{\mathbb{I}}}
\newcommand{\Id}{\operatorname{Id}}

\newcommand{\Gr}{\ensuremath{\mathrm{G}}}
\newcommand{\FK}{{\Bbbk}}

\newcommand{\ve}{\varepsilon}
\newcommand{\wt}{\widetilde}
\newcommand{\wa}{\overrightarrow}
\newcommand{\wb}{\overline}

\newcommand{\ms}[1]{\mbox{\tiny\ensuremath{#1}}}

\newcommand{\vp}{\varphi}

\newcommand{\Int}{\operatorname{Int}}

\newcommand{\kk}{\Bbbk}
\newcommand{\ideal}{\mathcal{I}} 
\newcommand{\Jideal}{\mathcal{J}} 
 
\newcommand{\Nideal}{\mathcal{N}} 

\newcommand{\mt}{{\operatorname{\mathsf{t}}}}

\newcommand{\qt}{\operatorname{\mathsf{t}}}
\newcommand{\Rib}{\operatorname{Rib}}
\newcommand{\ZZ}{\operatorname{Z}}
\newcommand{\LL}{\mathcal{L}}
\renewcommand{\SS}{\Sigma}
\newcommand{\Skein}{\mathcal{S}}
\newcommand{\Emb}{\mathrm{Emb}}
\newcommand{\Vect}{\mathrm{Vec}}
\newcommand{\atyp}{\operatorname{atyp}}
\newcommand{\Hp}{\mathscr{H}_{\ideal}}

\newcommand{\epsh}[2]
         {\begin{array}{c} \hspace{-1.3mm}
        \raisebox{-4pt}{\epsfig{figure=#1,height=#2}}
        \hspace{-1.9mm}\end{array}}
\newcommand{\epsw}[2]
         {\begin{array}{c} \hspace{-1.3mm}
        \raisebox{-4pt}{\epsfig{figure=#1,width=#2}}
        \hspace{-1.9mm}\end{array}}

\newcounter{exo} \newcounter{numexercice}
\renewcommand{\theexo}{\arabic{exo}}

\begin{document}

\title
{Admissible Skein Modules}

\begin{abstract}
  In this paper we introduce the notion of admissible skein modules
  associated to an ideal in a pivotal $\FK$--category.  We explain how
  these modules are generalizations of the Kauffman skein algebra and
  how they relate to renormalized quantum invariants coming from
  non-semisimple categories.
\end{abstract}

\author[F. Costantino]{Francesco Costantino}
\address{Institut de Math\'ematiques de Toulouse\\
118 route de Narbonne\\
 Toulouse F-31062}
\email{francesco.costantino@math.univ-toulouse.fr}

\author[N. Geer]{Nathan Geer}
\address{Mathematics \& Statistics\\
  Utah State University \\
  Logan, Utah 84322, USA}
\email{nathan.geer@gmail.com}

\author[B. Patureau-Mirand]{Bertrand Patureau-Mirand}
\address{Univ Bretagne Sud,  CNRS UMR 6205, LMBA, F-56000 Vannes, France}
\email{bertrand.patureau@univ-ubs.fr}

\maketitle
\setcounter{tocdepth}{3}

\section*{Introduction} In general, a skein module of a manifold $M$ is an algebraic object defined as a formal linear combination of embedded graphs in $M$, modulo local relations.  An important example of such modules is the Kauffman skein algebra of a surface introduced independently by Przyticki \cite{Pr1999} and  Turaev \cite{Tu1991b}.  It has a simple and combinatorial definition where the local relations are determined by the Jones polynomial or equivalently the  Kauffman bracket.  
In particular, the Kauffman  skein algebra $\Skein(S^2)$ associated to the 2-sphere is one dimensional and its dual $\Skein(S^2)^*$ is isomorphic to the linear span of the quantum trace; the natural pairing of these spaces recovers the Jones polynomial.  
 The simple definition of the Kauffman  skein algebra hides deep connections to many interesting objects like character varieties, TQFTs, quantum Teichm\"uller spaces, and many others, see for example \cite{Bu,BW2011,
Mull, Si2000, dTh}.  

Another active area of research is the study of the Renormalized Quantum Invariants (RQIs) of low-dimensional manifolds arising from non-semisimple categories, see for example \cite{BCGP14,CGP14,CGPT, GPT}.
The main purpose of this paper is to introduce the notion of an admissible skein module associated to an ideal in a general pivotal $\FK$--category.  This notion is a natural generalization of the Kauffman skein algebra allowing to recover RQIs instead of the Jones polynomials.
Like the Kauffman skein algebra we will discuss how these admissible skein modules have connections with interesting objects like TQFTs but now include new features  corresponding to non-semisimple categories.  

We will now briefly describe the main definition and results of the paper. 
 Let $\cat$ be an essentially small pivotal $\FK$--category.  Given an ideal $\ideal$ (a full subcategory of $\cat$ closed under tensor product and retracts) and an 
oriented 2-manifold $M$  we define the admissible skein module $\Skein_{\ideal}(M)$ as the
  $\FK$--span of $\ideal$-admissible ribbon graphs in $M$ modulo the
  span of $\ideal$--skein
  relations, see Definition~\ref{D:AdmisSkeinMod} (if $M$ is a 3-manifold then the same definition applies provided $\cat$ is ribbon).  Loosely speaking, an $\ideal$--skein
  relation is similar to a usual skein relation except that we require there is an edge colored with an object in $\ideal$ which is not completely contained in the local defining relation. A special case of this notion was considered in \cite{BCGP14}, where $\ideal$ is the ideal of projective modules over the unrolled quantum group $U_q^H(\mathfrak{sl}_2)$. 
  When $M=\Sigma$ is a surface, then  the   mapping class group of $\Sigma$ naturally acts on $\Skein_{\ideal}(\Sigma)$, see Proposition~\ref{P:MappingClassActs}.

  An important tool in the construction of RQIs is the notion of an m-trace on an ideal $\ideal$ in $\cat$.   Theorem~\ref{T:DiskRmt} says that the dual of the admissible skein module of the 2-sphere $\Skein_\ideal(S^2)^*$ is the $\FK$-span of m-traces on $\ideal$  (a related result was stated in a talk by Walker \cite{Walker}).   The pairing of this space with $\Skein_\ideal(S^2)$ gives the RQIs of links coming from these m-traces (generalizing the relationship of the Kauffman skein algebra, the Jones polynomial and the quantum trace).

The definition of the $\ideal$-admissible skein modules arises naturally in the setting of TQFT coming from (non) semisimple categories.  In particular,
in \cite{CGPV}, 
the authors with Virelizier construct a TQFT from any unimodular pivotal $\FK$-category with a non degenerate m-trace on an ideal $\ideal$ with some additional requirements; the $\ideal$-admissible skein module of a surface $\Sigma$ of this paper is exactly the TQFT space associated to $\Sigma$.
  Moreover, the underlying 3-manifold invariant of such a TQFT is related to the invariant of \cite{CGPT}, which is a generalization of the Turaev-Viro and unimodular Kuperberg   invariants.  
In the context of these TQFTs, Theorem \ref{T:ActionOfHandelbody} can be interpreted as saying that the morphism the TQFT associates to a handlebody is the linear form given in the theorem (a handlebody is a cobordism from the boundary of the handlebody to the empty surface.) Finally, Theorem \ref{T:finitewithoutboundary} says that when an ideal is graded finite then the skein module is graded and finite dimensional in each degree. 
  
In Section \ref{S:Examples} we consider three examples.  Remark when
$\ideal=\cat$ our definition of an admissible skein module reduces to
the usual definition.  When $\cat$ is semisimple our first example
relates admissible skein modules to standard TQFT spaces (as proven by
Bartlett in \cite{Bart2022}).  The second example is the category of
modules over the unrestricted quantum group; the skein modules are
graded and finite in each degree.  We expect that these skein modules
are related to the TV-type TQFT spaces of \cite{GP2013Top}.  More
generally, admissible skein modules are defined for non-unimodular
categories; it would be interesting to build TQFTs that extend the
mapping class group representation on these skeins.  The category of
modules over the Borel subalgebra of the unrestricted quantum group would be
particularly interesting.  In our final example, we explain that there
are nested ideals in categories of modules over certain Lie
superalgebras which all have non-trivial admissible skein modules.

There are many future directions to consider.  For example, there are categories which are not graded finite but have ``translation groups'' up to the action of which they become suitably finite.  It would be interesting to show that the admissible skein modules associated to these categories are finite dimensional.  An example of this is the skein algebra of thickened surfaces defined in \cite{BCGP14}.  In the same way we expect that in the case of quantum Lie superalgebras these admissible skein spaces should be related to a conjectural perturbative versions of super Chern-Simons theory, see \cite{GY2022, MikWitten} and references within.
  Finally, in the last few years several new important  results and ideas have appeared in the area of skein modules including in \cite{BBJ,BWT2022, CLe2021,CLe2022, FKBL2017, GJS, LeS2022}; it would be interesting to combine and connect these results and ideas to the skeins of this paper.  
 
{\bf Acknowledgements.} F.C. is grateful to Utah State University in Logan where this
work was started. He also acknowledges the funding from CIMI Labex ANR 11-LABX-
0040 at IMT Toulouse within the program ANR-11-IDEX-0002-02 and from the french ANR
Project CATORE ANR-18-CE40-0024.  F.C. would like to thank B. Haioun, D. Jordan, J. Korinman, T. Le for stimulating conversations.  The work of N.G.\ is partially supported by NSF grant DMS-2104497 and Institut de Math\'ematiques de Toulouse.

\section{Preliminaries}\label{S:Preliminaries}
\subsection{Pivotal categories}\label{SS:LinearCat}
  Let $\cat$ be a strict monoidal category with tensor product $\otimes$ and unit object
$\unit$.  The notation $V\in \cat$
means that $V$ is an object of $\cat$.
 
The category $\cat$ is a \emph{pivotal category} if it has duality morphisms
\begin{align*} \lcoev_{V} &: 
\unit \rightarrow V\otimes V^{*} , & \lev_{V} & :  V^*\otimes V\rightarrow
\unit , &
 \rcoev_V & : \unit \rightarrow V^{*}\otimes V, &  \rev_V & :  V\otimes V^*\rightarrow \unit
\end{align*}
which satisfy compatibility
conditions (see for example \cite{BW2, GKP2, Malt}). In particular, the left dual and right dual $f^*: W^*\to V^*$ of a morphism $f\colon V \to W$ in $\cat$ coincide:
\begin{align*}
f^*&= (\lev_W \otimes  \Id_{V^*})(\Id_{W^*}  \otimes f \otimes \Id_{V^*})(\Id_{W^*}\otimes \lcoev_V)= (\Id_{V^*} \otimes \rev_W)(\Id_{V^*} \otimes f \otimes \Id_{W^*})(\rcoev_V \otimes \Id_{W^*}).
\end{align*}
The category $\cat$ is \emph{ribbon} if it is pivotal, braided and has a twist satisfying compatibility
conditions (see for example \cite{GP2018}).

\subsection{$\FK$--categories}
Let $\FK$ be a field.  A \emph{$\FK$--category} is a category $\cat$
such that its hom--sets are finite dimensional $\FK$--vector spaces, the composition of morphisms
is $\FK$-bilinear, and the canonical $\FK$--algebra map
$\FK \to \End_\cat(\unit), k \mapsto k \, \Id_\unit$ is an
isomorphism.  A \emph{monoidal $\FK$--category} is a monoidal category
$\cat$ such that $\cat$ is a $\FK$--category and the tensor product of
morphisms is $\FK$-bilinear.

\subsection{M-traces on ideals in pivotal categories}\label{SS:trace}
Let $\cat$ be a pivotal $\FK$--category.  Here we recall the definition
of an m-trace on an ideal in $\cat$, for more details see
\cite{GKP2,GPV}.  By an \emph{ideal} of $\cat$ we mean a full
subcategory, $\ideal$, of~$\cat$ which is
\begin{description}
\item[Closed under tensor product]
 If $V\in \ideal$ and
  $W\in \cat$, then $V\otimes W$ and $W \otimes V $ are objects of
  $\ideal$.
\item[Closed under retractions] If $V\in \ideal$, $W\in \cat$
  and there are morphisms $f:W\to V$, $g:V\to W$ such that
  $g f=\Id_W$, then $W\in \ideal$ (we say $W$ is a \emph{retract} of $V$).
\end{description}
A \emph{right (resp.\ left) partial trace} of $f\in \End_\cat(U\otimes W)$ is the morphism
$$\ptr_R^W(f)=(\Id_U \otimes \rev_W)(f\otimes \Id_{W^*})(\Id_U \otimes \lcoev_W), 
 \text{ resp.\ } \ptr_L^U(f)=
( \lev_U\otimes \Id_W )( \Id_{U^*} \otimes f)( \rcoev_U\otimes \Id_W).
$$
A \emph{right m-trace} on an ideal $\ideal$ is a family of linear functions
$\{\qt_V:\End_\cat(V)\rightarrow \FK \}_{V\in \ideal}$
such that:
\begin{description}
\item[Cyclicity] 
 If $U,V\in \ideal$ and $f:V\rightarrow U,$ $g:U\rightarrow V$ are any morphisms in $\cat$ then
$
\qt_V(g f)=\qt_U(f  g).$
\item[Right partial trace property]  If $U\in \ideal$, $W\in \cat$, $f\in \End_\cat(U\otimes W)$ then
$
\qt_{U\otimes W}\bp{f}
=\qt_U(\ptr_R^W(f)).
$
\end{description}
A \emph{left m-trace} on $\ideal$ is a family of linear functions $\{\qt_V:\End_\cat(V)\rightarrow \FK \}_{V\in \ideal}$ which is cyclic and  satisfies
\begin{description}
 \item[Left partial trace property] 
  If $U\in \ideal$, $W\in \cat$, $g\in \End_\cat(W\otimes U)$ then
$
\qt_{W\otimes U}\bp{g}
=\qt_U(\ptr_L^W(g)).
$
\end{description}
A \emph{m-trace} on an ideal is right m-trace which is also a left m-trace.
\begin{remark}\label{rk:ideal-gen}
  Since an intersection of ideals is an ideal, we can define the ideal
  generated by a collection of objects as the smallest ideal that
  contains these objects.
\end{remark}

\subsection{Invariants of colored ribbon graphs}
Let $\cat$ be a pivotal $\FK$--category.
  
\begin{minipage}{0.7\linewidth}
  A morphism
$f:V_1\otimes{\cdots}\otimes V_n \rightarrow W_1\otimes{\cdots}\otimes
W_m$ in $\cat$ can be represented by a box and arrows, which are called {\it coupons}, see graph on the right:
\end{minipage}
\begin{minipage}{0.3\linewidth}
\vspace*{-5pt}
$$ 
\epsh{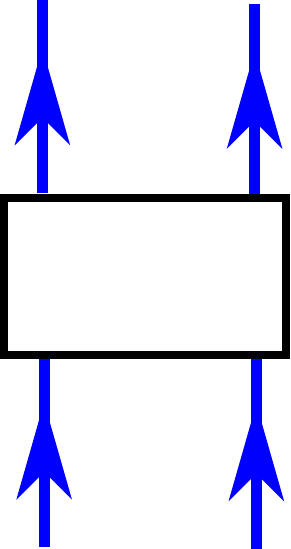}{9ex}\put(-32,-14){{\footnotesize $V_1$}}\put(1,-14){{\footnotesize $V_n$}}\put(-15,-16){{\footnotesize $\cdots$}}
\put(-32,14){{\footnotesize $W_1$}}\put(1,14){{\footnotesize $W_n$}}\put(-15,16){{\footnotesize $\cdots$}}\put(-13,0){{\footnotesize $f$}}
$$
\end{minipage} 
All coupons have  top and  bottom sides which in our  pictures will be the horizontal sides of the coupons.   By a ribbon graph in an
oriented manifold $\SS$, we mean an
oriented compact surface embedded in $\SS$ which is decomposed into
elementary pieces: bands, annuli, and coupons (see \cite{Tu}) and is a thickening of an oriented graph.  
In particular, the
vertices of the graph lying in $\Int \SS=\SS\setminus\partial \SS$ are
thickened to coupons.  A $\cat$--colored ribbon graph is a ribbon graph
whose (thickened) edges are colored by objects of $\cat$ and whose
coupons are colored by morphisms of~$\cat$.  
The intersection of a $\cat$--colored ribbon graph in $\Sigma$ with $\partial \Sigma$ is required to be empty or to consist only of vertices of valency~1.  
When $\SS$ is a surface the ribbon graph is just a tubular neighborhood of the graph.

A $\cat$--colored ribbon graph in $\R^2$ (resp.\ $S^{2}=\R^2\cup\{\infty\}$) 
is
called \emph{planar} (resp.\ \emph{spherical}).
Let $\Rib_\cat$ be the category of planar $\cat$--colored ribbon graphs and let $F:\Rib_\cat\to \cat$ be the pivotal functor
associated with $\cat$ via  the Penrose graphical calculus, see for example \cite{GPV}; note $F$ does not preserve the duality.   
 Let $\LL_{adm}$ be the class of all spherical $\cat$--colored ribbon graphs obtained as  the braid closure of a (1,1)-ribbon graph $T_V$ whose open edge is colored with  an object $V\in\ideal$.
  Given an m-trace $\mt$ on $\ideal$ we can renormalize $F$ to an invariant 
\begin{equation}\label{E:DefF'}
F':\LL_{adm}\to \kk \text{ given by } F'(L)=\mt_V(F(T_V))
\end{equation}
where $T_V$ is any (1,1)-ribbon graph whose closure is $L\in \LL_{adm}$.  The properties of the m-trace imply that $F'$ is a isotopy invariant of $L$, see \cite{GPV}.  

If $\cat$ is a ribbon category then let $\LL_{adm}$ be the class of all $\cat$--colored ribbon graphs in $\R^3$ obtained as  the braid closure of a (1,1)-ribbon graph $T_V$ whose open edge is colored with  an object $V\in\ideal$.  Then the assignment $F'$ given in Equation \eqref{E:DefF'}, where $F$ is the R-T functor,  is an isotopy invariant of $L$ in $\R^3$.

\section{Admissible Skein Modules}
\subsection{Definition of admissible skein modules}\label{SS:DefAdmSkein}
Let $\cat$ be an essentially small pivotal $\FK$--category with an ideal $\ideal$.
Let $M$ be a
oriented manifold of dimension two or three. 
In this paper all manifolds will be oriented and if we consider the skein module of a three dimensional manifold we require $\cat$ to be ribbon.

An \emph{$\ideal$--admissible ribbon graph} in $M$ is a
$\cat$--colored ribbon graph in $M$ where each connected component of
$M$ contains at least one edge colored with an object in $\ideal$.

A \emph{box} in a manifold $M$ of dimension $n=2,3$ is a proper embedding of $[0,1]^{ n} $ in $M$.  Its top and bottom faces are $[0,1]^{n-1}\times \{1\}$ and $[0,1]^{n-1}\times \{0\}$, respectively.  A ribbon graph is \emph{transverse} to a box if it intersects the box transversally on the lines $[0,1] \times \{0,1\}$ if $n=2$ and  $[0,1] \times \{\frac12\} \times  \{0,1\}$ if $n=3$.

Let $(a_1,\ldots,a_m)\in\kk^m$ and
$\Gamma_1,\,\Gamma_2,\ldots,\Gamma_m$ be $\cat$--colored ribbon graph
in $M$.  We say that the formal linear combination $\sum_ia_i\Gamma_i$
is a \emph{$\ideal$--skein relation} if there exists a box $B$ in $M$ and for $i=1,...,m$
there exists a graph $\Gamma'_i$ which is isotopic to $\Gamma_i$, transverse to the
box $B$ and identical outside $B$ (i.e. $\Gamma'_i\cap (M\setminus B)=\Gamma'_j\cap (M\setminus B)\ \forall i,j$) such that (1)
$\sum_ia_iF(\Gamma'_i\cap B)=0$ as a homomorphism in $\cat$ and (2) each $\Gamma'_i$ has an
$\ideal$--colored edge not completely contained in $B$. Two linear combinations are \emph{$\ideal$--skein equivalent} if their difference is an $\ideal$--skein relation.

\begin{definition}\label{D:AdmisSkeinMod}
  Define the \emph{admissible skein module} $\Skein_{\ideal}(M)$ as
  the $\FK$--span of $\ideal$--admissible ribbon graphs in $M$ modulo
  the span of $\ideal$--skein relations.
\end{definition}
 Since  $\cat$ is essentially small then $\Skein_{\ideal}(M)$ is always a set; indeed up to skein equivalence we can replace any color with a fixed representative of its  isomorphism class in $\cat$.

Let $\Emb_n$ be the category whose objects are oriented $n$-dimensional manifolds and morphisms are isotopy classes of orientation preserving 
proper embeddings. 
The following holds:

\begin{prop}\label{P:MappingClassActs}
Let $n\in \{2,3\}$. We assume that if $n=3$ then $\cat$ is ribbon. We have  $\Skein_{\ideal}:\Emb_n \to \Vect$ is a functor. In particular, if $n=2$ they provide representations of the mapping class group of surfaces. 
\end{prop}
\begin{proof}
If $f:M\to M'$ is an embedding and $\Gamma\in \Skein_{\ideal}(M)$ then $f(\Gamma)\in \Skein_{\ideal}(M')$. Furthermore, the image of a skein relation under $f$  is a skein relation.  
\end{proof}

\begin{prop}\label{P:FunctorIndSkein}
Let $\cat$ and $\catd$ be essentially small pivotal $\FK$--categories.   Let $G: \cat\to\catd$ be a strict pivotal functor (see \cite[Section 1.7.5]{TuVirelizierBook}) and let $\ideal$ be an ideal of $\cat$.  Let $\Jideal$ be the ideal of $\catd$ generated by the objects of $G(\ideal)$.  The functor $G$ induces a natural transformation $G_*:\Skein_\ideal\to \Skein_\Jideal$.
\end{prop}
\begin{proof}
It is clear that the image of an $\ideal$--admissible graph is a $\Jideal$-admissible graph.
Furthermore, a $\ideal$-skein relation is transformed via the application of $G$ into a $\Jideal$-skein relation; indeed $F_\catd\circ G_*=G\circ F_\cat:\Rib_\cat\to \catd$ and if the complement of a box contains a $\ideal$-colored edge then after application of $G_*$ the complement of the box contains a $\Jideal$-colored edge.
\end{proof}

\begin{prop}\label{prop:ddd}
  $\Skein_\ideal(M)$ is generated by ribbon graphs where each edge is
  colored by an object of $\ideal$.
\end{prop}
\begin{proof}
  By adding identity coupons it is easy to see $\Skein_\ideal(M)$ is
  generated by ribbon graphs where each edge is joining two different
  coupons.  Then we can induct on the number of edges whose color
  does not belongs to $\ideal$.\\
  \begin{minipage}{0.6\linewidth}
    Using a skein relation, we can
  replace such an edge colored by $V$ which is close to an
   edge colored by $U\in \ideal$ with an edge colored by $V\otimes U\in\ideal$
   changing the coupons as in the following figure:
\end{minipage}
  \begin{minipage}{0.4\linewidth}
    $$\epsh{fig4a}{16ex}\put(1,2){\ms{U}}\put(-25,2){\ms{V}}
   \put(-19,22){\ms{f}}\put(-19,-19){\ms{g}}
   \qquad\longrightarrow\quad
   \epsh{fig4b}{16ex}\put(-11,2){\ms{V\!\!\otimes\! U}}
   \put(-25,22){\ms{f\!\otimes\!\Id_U}}\put(-25,-20){\ms{g\!\otimes\!\Id_U}}$$
  \end{minipage}
      Remark that up to skein relations in the neighborhood of the
   coupons adjacent to the $V$-colored edge, we can always assume that
   its position relatively to the coupons is as in the above figure.
 \end{proof}

\subsection{Algebra and module structures of $\ideal$-skein modules}

\begin{prop}\label{prop:boundaryaction}
Let  $\cat$ be an essentially small ribbon $\FK$--category and $\ideal$ and $\Jideal$ be  ideals of $\cat$.  If  $\Sigma$ is a surface,  then $\Skein_{\ideal}(\Sigma)$ is an associative algebra, which is non unital if $\ideal\neq \cat$.  Furthermore,  if $M$ is a $3$-manifold and $i:\Sigma\hookrightarrow \partial M$ is an orientation preserving (resp. reversing) embedding, then $\Skein_{\ideal}(M)$ is a left (resp. right) module over $\Skein_{\Jideal}(\Sigma)$.
\end{prop}

\begin{proof}
The algebra operation of $\Skein_{\ideal}(\Sigma)$ is given by vertical stacking in $\Sigma\times [0,1]$: if $\Gamma,\Gamma'\subset \Sigma$ are ribbon graphs then $\Gamma\cdot \Gamma'$ is the ribbon graph obtained by embedding $\Gamma$ in $\Sigma\times \{1\}$ and  $\Gamma'$ in $\Sigma\times \{0\}$ such that their projections in $\Sigma$ are transverse and then replacing all the crossings by coupons decorated by the braiding of $\cat$ in order to get a ribbon graph in $\Sigma$. 

We claim that this operation gives a well defined associative product.  Indeed if $\Gamma$ and $\Gamma'$ are $\ideal$--admissible graphs then so is $\Gamma\cdot \Gamma'$. Furthermore if $\Gamma'$ and $\Gamma''$ are related by a $\ideal$-skein relation on a box $B$ then we will show that $\Gamma\cdot \Gamma'$ and $\Gamma\cdot \Gamma''$ are $\ideal$-skein equivalent:  up to isotopy of $\Gamma$ we can suppose that both $\Gamma\cdot \Gamma'$ and $\Gamma\cdot \Gamma''$ intersect $B$ transversally and that $\Gamma\cap B$ lies entirely at the left of both $\Gamma'\cap B$ and $\Gamma''\cap B$.  Then $F(\Gamma\cdot \Gamma'\cap B)=F(\Gamma\cap B)\otimes F(\Gamma'\cap B)$ and $F(\Gamma\cdot \Gamma''\cap B)=F(\Gamma\cap B)\otimes F(\Gamma''\cap B)$.
But since $F(\Gamma'\cap B)=F(\Gamma''\cap B)$ we conclude that $F(\Gamma\cdot \Gamma'\cap B)=F(\Gamma\cdot \Gamma''\cap B)$.
The operation is clearly associative. 
If $\unit\notin \ideal$ then the  empty graph is not a $\ideal$--admissible graph so that $\Skein_{\ideal}(\Sigma)$ is not unital. 

Suppose $i:\Sigma\hookrightarrow \partial M$ is an orientation preserving embedding.  
We will show the assignment $\cdot : \Skein_{\Jideal}(\Sigma)\times \Skein_{\ideal}(M)\to \Skein_{\ideal}(M)$
given by $[\alpha]\cdot [\beta]=[i(\alpha)\sqcup \beta]$ is well defined where $\alpha$ and $\beta$ are representatives of $[\alpha]\in \Skein_{\Jideal}(\Sigma)$ and $[\beta]\in \Skein_{\ideal}(M)$, respectively (here we assume $\beta\subset int(M)$ which can be done by pushing it inside $M$ with an isotopy). 
Notice $i(\alpha)\sqcup \beta$ is an $\ideal$-admissible graph in $M$.  If $\beta$ and $\beta'$ are related by a $\ideal$-skein relation  given in a box $B$ then up to isotopy we can assume $B\subset int(M)$ implying  that $i(\alpha)\sqcup \beta$ and $i(\alpha)\sqcup \beta'$ are $\ideal$--skein related in $B$.  
Also, if $\alpha$ and $\alpha'$ are related by a $\Jideal$-skein relation given by a box $B\subset \Sigma$ then $i(\alpha)\sqcup \beta$ and $i(\alpha')\sqcup \beta$ are $\Jideal$-related in a box $B'$ obtained by thickening slightly $i(B)$ inside $M$.  Since $[\beta]\in \Skein_{\ideal}(M)$ then up to isotopy this relation is actually a $\ideal$-skein relation.  
Therefore we have a well defined action as claimed.  If the orientation induced by $M$ on $i(\Sigma)$ is negative then the right action is proven similarly where $[\beta\cdot \alpha]=[\beta\sqcup i(\alpha)]$ and $i(\alpha)$ lies below $\beta$.  
\end{proof}
\section{Skeins of elementary spheres and balls}
In this section we will interpret the skeins of the disk and the sphere in terms of m-traces. 
K. Walker and D. Reutter announced a related result, see \cite{Walker}. 
\begin{theorem}\label{T:DiskRmt}Let $\cat$ be an essentially small pivotal $\FK$--category and let $\ideal$ be an  ideal of $\cat$.  Then
$$\Skein_\ideal(D^2)^*\cong \{\text{right m-traces on } \ideal\}\cong \{\text{left m-traces on } \ideal\} \;\; \text{ and } \;\;
\Skein_\ideal(S^2)^*\cong \{\text{m-traces on } \ideal\}$$
where $D^2$ and $S^2$ are the 2-disk and 2-sphere, respectively.
\end{theorem}
\begin{proof}
  Here we prove the right version of the first statement, the left
  version is analogous.  Given $T\in \Skein_\ideal(D^2)^*$ then define
  the family of linear functions
  $\{\mt^T_V:\End_\cat(V)\to \FK\}_{V\in \ideal}$ as follows.
  For $f\in \End_\cat(V)$ let $\mt^T_V(f)=T(O_f)$ where $O_f$ is the
    ribbon graph which is the right closure of the coupon colored with
    $f$.  \\
  \begin{minipage}{0.4\linewidth}
    Since the coupon colored with $f\circ g$ is $\ideal$-skein
    equivalent to the coupon colored with $f$ composed with the coupon
    colored with $g$, it follows that $O_{f\circ g}$ is skein
    equivalent to $O_{g\circ f}$, via an isotopy which exchanges $f$
    and $g$:
  \end{minipage}
  \begin{minipage}{0.6\linewidth}
    $$\epsw{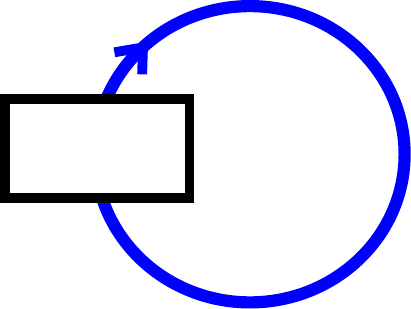}{2cm}\put(-51,2){\ms{g\circ f}}
    =\epsw{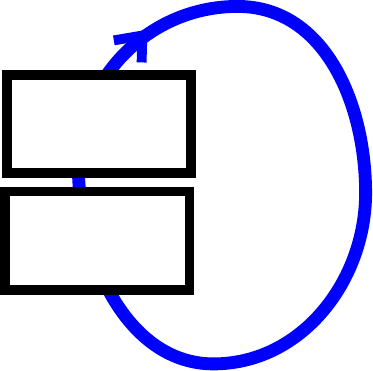}{2cm}\put(-45,11){\ms g}\put(-45,-9){\ms f}
    =\epsw{Ograph2b}{2cm}\put(-45,11){\ms f}\put(-45,-9){\ms g}
    =\epsw{Ographb}{2cm}\put(-51,2){\ms{f\circ g}}$$
  \end{minipage}
  Therefore, the family $\{\mt^T_V\}_{V\in \ideal}$ satisfies the
  cyclicity property of an m-trace.  \\
  \begin{minipage}{0.6\linewidth}
  Next, let $W\in \cat$ and
  $f\in \End_\cat(V\otimes W)$ then $O_{f}$ is skein equivalent to the
  closer of the coupon colored with $f$ with two incoming and outgoing
  edges colored with $V$ and $W$ :
 \end{minipage}
  \begin{minipage}{0.4\linewidth}$$\epsw{Ographb}{2cm}\put(-46,2){\ms{ f}}
 =\epsw{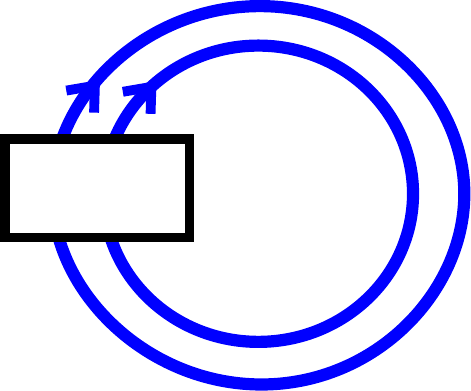}{2.2cm}\put(-51,2){\ms{ f}}$$
\end{minipage}
 This shows that the family $\{\mt^T_V\}_{V\in \ideal}$ satisfies the
 right partial trace property of an m-trace.

 Conversely, given a right m-trace $\mt$ on $\ideal$ we can define an
 element of $F'_\mt\in\Skein_\ideal(D^2)^*$ as follows.
 Let $\Gamma$ be an $\ideal$--admissible ribbon graph in $D^2$.  A
 \emph{cutting path} for $\Gamma$ is any embedding
 $\gamma:[0,1]\to D^2$ starting from a boundary point of $D^2$ and
 ending in any point in the interior of $D^2\setminus \Gamma$ such
 that the following conditions hold: a) $\gamma$ does not meet any
 coupon of $\Gamma$, b) $\gamma$ is transverse to the edges of
 $\Gamma$ and c) $\gamma$ intersects at least one $\ideal$-colored
 edge of $\Gamma$.  The complement of a tubular neighborhood of
 $\gamma$ is then a box
 $B_\gamma$
 whose bottom and top correspond to the left and right side of
 $\gamma$, respectively. 
The image by the RT--functor of the graph 
$\Gamma_\gamma$ in $B_\gamma$ is
$F(\Gamma_\gamma)\in\End_\cat(V_\gamma)$ where
$V_\gamma\in\ideal$ because $\gamma$ intersect an
$\ideal$ colored edge.  Then set
$F'_\mt(\Gamma)=\mt_{V_\gamma}(F(\Gamma_\gamma))$.
Next we prove that $F'_\mt(\Gamma)$ is independent of
$\gamma$.  We do this in the following two steps.
\vspace{2pt}

\noindent
\textbf{Step 1}.  Assume $\gamma$ and $\gamma'$ are disjoint cutting paths and prove $\mt_{V_\gamma}(F(\Gamma_\gamma))=\mt_{V_{\gamma'}}(F(\Gamma_{\gamma'}))$.

We first locally modify $\Gamma$ so that all its intersection points
with $\gamma$ and $\gamma'$ are positive:
Away from the tubular neighborhood of $\gamma\cup\gamma'$ the graph
$\wt\Gamma$ is just $\Gamma$.  \\
\begin{minipage}{0.5\linewidth}
  For each intersection point $p$ of
$\gamma$ or $\gamma'$ and $\Gamma$, let $e$ be a small segment of the
edge of $\Gamma$ near $p$.  If the orientation of this intersection is
negative (with respect to the orientation of the $D^2$) then we
replace $e$ with a segment containing two coupons colored with
identities joined by an edge crossing $\gamma$ or $\gamma'$ positively
and colored by the dual color of $e$:
\end{minipage}
\begin{minipage}{0.5\linewidth}
$$\epsw{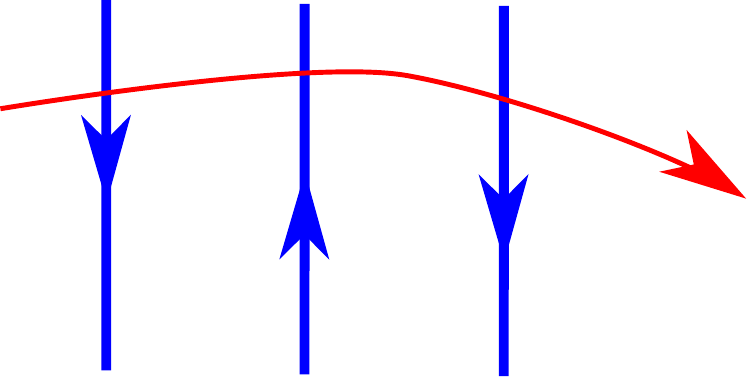}{18ex}\put(-4,7){\ms{\gamma}}
\put(-74,0){\ms{V_1}}\put(-44,0){\ms{V_2}}\put(-22,0){\ms{V_3}}
\qquad\longrightarrow\qquad\epsw{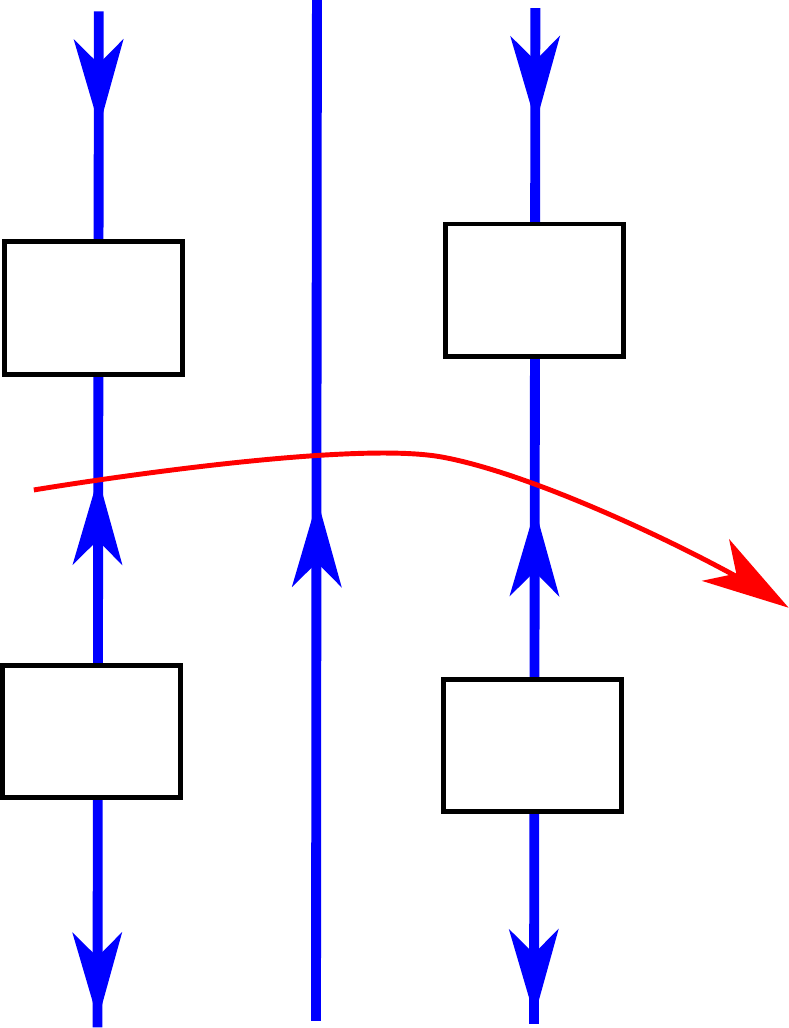}{18ex}
\put(-79,-4){\ms{V_1^*}}\put(-44,0){\ms{V_2}}\put(-23,-10){\ms{V_3^*}}
\put(-73,20){\ms{\Id}}\put(-28,-24){\ms{\Id}}
\put(-73,-23){\ms{\Id}}\put(-28,21){\ms{\Id}}
$$
\end{minipage}
Clearly $F(\Gamma_\gamma)=F(\wt\Gamma_\gamma)$ and $F(\Gamma_{\gamma'})=F(\wt\Gamma_{\gamma'})$ so that up to replacing $\Gamma$ with $\wt\Gamma$, we can assume all crossing of $\gamma$ or $\gamma'$ with $\Gamma$ are positive.

In this case, the complement of a tubular neighborhood of $\gamma\cup\gamma'$ is a box $B_{\gamma\cup\gamma'}$ with an obvious graph $\Gamma_{\gamma\cup\gamma'}$ 
whose left closure is $\Gamma_\gamma$ and right
closure is the $\pi$-rotation of $\Gamma_{\gamma'}$ in 
the $\pi$-rotation of $B_{\gamma'}$, we denote the later graph by $\operatorname{rot_\pi}(\Gamma_{\gamma'})$.  Then $F(\operatorname{rot_\pi}(\Gamma_{\gamma'}))=F(\Gamma_{\gamma'})^*\in \End_\cat(V_{\gamma'}^*)$.

  Set $g=F(\Gamma_{\gamma\cup\gamma'})\in \End_\cat(V_{\gamma'}^*\otimes V_{\gamma})$ then
  \begin{align*}
\label{}
  \mt_{V_\gamma}(F(\Gamma_\gamma))=  \mt_{V_\gamma}\Big(\ptr_L^{V_{\gamma'}^*}(g)\Big)
=\mt_{V_{\gamma'}^{**}}\Big((\ptr_R^{V_\gamma}(g))^*\Big)
    =\mt_{V_{\gamma'}^{**}}\Big(\big(F(\operatorname{rot_\pi}(\Gamma_{\gamma'}))\big)^*\Big)
   =\mt_{V_{\gamma'}}(F(\Gamma_{\gamma'}))
  \end{align*}
where the first and third equalities come from the definition of $\Gamma_{\gamma\cup\gamma'}$ , the second equality follows from Lemma~4.b of \cite{GPV} which can be restated as
$\mt_{U}(\ptr_L^{W^*}(g))=  \mt_{W^{**}}((\ptr_R^{U}(g))^*)$
for all $W^*,U\in \ideal$ and $g\in \End_\cat(W^*\otimes U)$ and the last equality follows from $\mt_{V^{**}}(f^{**})=\mt_V(f)$ for any $f\in\End_\cat(V)$. \vspace{2pt}

\noindent
\textbf{Step 2}.  If $\gamma$ and $\gamma'$ are cutting paths with less than $n$ intersection points then
$\mt_{V_\gamma}(F(B_\gamma))=\mt_{V_{\gamma'}}(F(B_{\gamma'}))$.

We prove this step by induction on $n$.  The base case is done in Step 1.   Assume the statement is true for $n$ and let $\gamma$ and $\gamma'$ be two cutting paths intersecting $n$ times.  
We claim that there exists a cutting path $\alpha$ intersecting each of $\gamma$ and $\gamma'$ less than $n$ times, then by induction we have $\mt_{V_\gamma}(F(\Gamma_\gamma))=\mt_{V_\alpha}(F(\Gamma_\alpha))=\mt_{V_{\gamma'}}(F(\Gamma_{\gamma'}))$.

Indeed let $\gamma''$ be the sub-arc of $\gamma$ going from $\partial D$ to the first edge of $\Gamma$ colored by an object of $\ideal$ and then crossing this edge of a small arc.
It is clear that $\gamma''$ is a cutting path (note by construction it  intersects an $\ideal$-colored edge of
$\Gamma$).   If $\gamma''$ intersects  $\gamma'$ less than $n$ times we can push $\gamma''$ slightly to either side of $\gamma$ to obtain the cutting path $\alpha$.  In this case, we have constructed a cutting path as required for the induction step.   If $\gamma''$ intersects $\gamma'$ exactly $n$ times, let $p$ be the last intersection (in the orientation of $\gamma''$) between $\gamma''$ and $\gamma'$. Consider the arc obtained by following $\gamma'$ until getting to $p$ and then following $\gamma''$ until its end. By pushing this arc slightly to its  right or left (according to the sign of the intersection at $p$ between $\gamma$ and $\gamma'$) we get a cutting arc  $\alpha$  intersecting each $\gamma$ and $\gamma'$ less than $n$ times.   This completes Step 2 and shows that $F'_\mt(\Gamma)$  is independent of $\gamma$.
   
Finally $F'_\mt$ is invariant by $\ideal$--admissible skein relation:
indeed one can always find a cutting path avoiding any box involved in
an $\ideal$--admissible skein relation.  
To summarize, we showed that giving an m-trace $\mt$ we can construct an element $F'_\mt\in\Skein_\ideal(D^2)^*$ and a converse construction.  It is clear the two constructions are inverse of each other thus proving the first isomorphism of the theorem.

Let now consider the spherical case.  Any linear form $T$ on
$\Skein_\ideal(S^2)$ induces a right m-trace $\mt^T$ defined on
morphisms by $\mt^T_V(f)=T(O_f)$ as above (except that the graph $O_f$
is now in $S^2$).  Since $O_f$ and $O_{f^*}$ are isotopic in $S^2$,
this right m-trace is equal to its dual and \cite[Lemma 3]{GPV} implies this is an m-trace.

Reciprocally, let $\mt$ be a m-trace on $\ideal$ thus it is in
particular a right m-trace.  If $\Gamma$ is an $\ideal$--admissible
ribbon graph in $S^2$ and $p\in S^2\setminus \Gamma$ let
$F'_\mt(\Gamma,p)=F'_\mt(\Gamma')$ where
$\Gamma'\in\Skein_\ideal(D^2)$ is in the interior of
$D^2\simeq S^2\setminus\{p\}$. We claim that $F'_\mt(\Gamma,p)$ does
not depend of $p$: A cutting path in $S^2$ for $\Gamma$ is any path
$\gamma$ starting from a point $p_1\notin\Gamma$ and finishing at a
point $p_2\notin\Gamma$, meeting no coupon of $\Gamma$, transverse to
the edges of $\Gamma$ and intersecting at least one $\ideal$-colored
edge of $\Gamma$.  Let $\stackrel\leftarrow \gamma$ be the inverse
path from $p_2$ to $p_1$.  The path $\gamma$ induces a cutting path in
$D^2$ to compute $F'_\mt(\Gamma,p_1)$ and similarly
$F'_\mt(\Gamma,p_2)$ can be computed using the cutting path in $D^2$
induced by $\stackrel\leftarrow \gamma$.  But the two boxes obtained
by cutting along $\gamma$ and $\stackrel\leftarrow \gamma$ are related
by a $\pi$-rotation and since $\mt$ is an m--trace, we have
$F'_\mt(\Gamma,p_1)=F'_\mt(\Gamma,p_2)$. Hence
$F'_\mt(\Gamma)=F'_\mt(\Gamma,p)$ is well defined by choosing any
point $p$.  Finally if an $\ideal$--admissible skein relation in $S^2$
is defined with a box $B$, we can choose for all graphs involved have the
same point $p$ outside the box so the skein relation comes from a
skein relation in $D^2$ on which $F'_\mt$ vanishes.
\end{proof}
\begin{corollary}\label{T:dim3SphereSkein}
  Let $\cat$ be an essentially small ribbon $\FK$--category and let
  $\ideal$ be an ideal of $\cat$.  Then
  $$\Skein_\ideal(B^3)^*\cong \Skein_\ideal(S^3)^*
  \cong \{\text{m-traces on } \ideal\}$$
  where $B^3$ and $S^3$ are the 3-ball and 3-sphere, respectively.
\end{corollary}
\begin{proof}
  We will show that $\Skein_\ideal(D^2)$ is isomorphic to
  $\Skein_\ideal(B^3)$, then by \cite[Remark 9]{GPV} since any right
  m-trace is a m-trace in a ribbon category then the corollary will
  follow from Theorem \ref{T:DiskRmt}.

  The inclusion of $D^2$ into $B^3$ induces a morphism
  $\Skein_\ideal(D^2)\to\Skein_\ideal(B^3)$.  It has an inverse map
  given by sending any ribbon graph in $B^3$ to a standard projection
  in $D^2$ at the price of replacing crossings with coupons colored by
  the braiding.  Any two such projections are related by standard
  Reidemeister type moves which are skein equivalences in $D^2$.
  Finally a skein relation in $B^3$ can be isotoped to project on a
  skein relation in $D^2$.
\end{proof}

A \emph{cyclic-trace} is a family of linear functions
$\{h_V:\End_\cat(V)\rightarrow \FK \}_{V\in \ideal}$
such that if $U,V\in \ideal$ then for any morphisms $f:V\rightarrow U $ and $g:U\rightarrow V$  in $\cat$ we have 
$
h_V(g f)=h_U(f  g).$

\begin{theorem}
Let $\cat$ be an essentially small ribbon $\FK$--category and let
  $\ideal$ be an ideal of $\cat$.  Then 
$$
\Skein_{\ideal}(S^1\times[-1,1])^*
  \cong \{\text{cyclic-traces on } \ideal\}.
  $$
  \end{theorem}
\begin{proof}
Given a $T\in \Skein_{\ideal}(S^1\times[-1,1])^*$  then define the family of linear functions $\{h^T_V:\End_\cat(V)\to \FK\}_{V\in \ideal}$ as follows.  For $f\in \End_\cat(V)$ let $\mt^T_V(f)=T(O_f)$ where $O_f$ is the ribbon graph which is the right closure in the annulus of the coupon colored with $f$.   Arguing as in the proof of Theorem \ref{T:DiskRmt} this family is a cyclic-trace.  

Conversely, given a cyclic-trace $h$ on $\ideal$ we can define an
element of $F'_h\in\Skein_\ideal(S^1\times[-1,1])^*$ as follows.  Let $\Gamma$
be an $\ideal$--admissible ribbon graph in $S^1\times[-1,1]$.  We say a radius is an arc in $S^1\times[-1,1]$ of the form $\theta \times [-1,1]$ such that 1) it intersects $\Gamma$ transversally, 2) it does not intersect any coupons of $\Gamma$ and 3) it intersects at least one $\ideal$-colored edge of $\Gamma$.  Up to isotopy of $\Gamma$ it is easy to see that such radius exist.  

To each radius we can associate an object $V\in \ideal$ and a morphism $f:V\to V$ as follows.  Starting at $(\theta,-1)$ and moving along the radius we can define $V$ as the tensor product of the colors of the edges of $\Gamma$ intersecting the radius or it dual depending on the sign of the intersection.  The intersection of $\Gamma$ with the complement of a neighborhood of the radius represents, via the R-T functor, and endomorphism $f:V\to V$.  Define $F'_h(\Gamma)=h(f)$.  

We claim that $F'_h(\Gamma)$ does not depend on the choice of the radius.  Indeed, if $\theta' \times [-1,1]$ is another radius then let $V'$ and $f'$ be its associated color and morphism.  Cutting as above both radii we can construct two morphisms:  $f_1:V\to V'$ and $f_2:V'\to V$ such that $f=f_2\circ f_1$ and $f'=f_1\circ f_2$.  Then by the cyclicity property $h(f)=h(f')$.  This proves the claim.

Finally, if $\Gamma$ and $\Gamma'$ are skein equivalent then we will show that $F'_h(\Gamma)=F'_h(\Gamma')$.  Since we consider $\ideal$-skein relations, up to isotopy we can assume the box supporting the skein relation does not intersect a common radius for both  $\Gamma$ and $\Gamma'$.  Therefore, the morphisms $f$ and $f'$ associated to $\Gamma$ and $\Gamma'$ by that radius are equal by definition of the skein relation.  This proves the claim.  
\end{proof}
\begin{remark}
The skein module of $\Skein_{\ideal}(S^1\times[-1,1])$ has a natural algebra structure induced by the embedding of two copies of $S^1\times[-1,1]$ into $S^1\times [-1,1]$ as $S^1\times[-1,0]$ and $S^1\times [0,1]$. This algebra is in general not unital unless $\ideal=\cat$, in which case the empty skein represents the unit. 
\end{remark}

\section{Higher genus surfaces and handlebody graphs invariants}

\subsection{Handlebodies and linear functionals on skein modules}

Let $\cat$ be a pivotal $\FK$--category and $\ideal$ an ideal in $\cat$.    An m-trace $\mt$ is  \emph{non-degenerate} if for any object
$P\in\ideal$, the pairing 
$$
\Hom_\cat(\unit, P) \times \Hom_\cat(P,\unit) \to \FK , \ \   (x,y)\mapsto \mt_P(xy)
$$
 is non-degenerate.  
For such an m-trace, pick a basis $\{x_i\}$ of $\Hom_\cat(\unit, P)$ and  let $\{y_i\}$ be the dual basis of
$\Hom_\cat(P,\unit)$ with respect to the above pairing.

By a \emph{multi-handlebody} we mean a disjoint union of a finite number of oriented 3-dimensional handlebodies.  
An $\ideal$--admissible ribbon graph $\Gamma$ on a multi-handlebody $H=\sqcup H_j$ is a $\cat$-colored ribbon graph in $\partial H=\sqcup \partial H_j$ such that for each $j$ we have $\partial H_j\cap \Gamma$ is  a non-empty $\ideal$--admissible ribbon graph in the boundary of the handlebody $H_j$.  
Let 
$ \Hp$ be the class of all pairs $(H,\Gamma)$ where $H$ is a multi-handlebody and $\Gamma$ is a linear combination of  $\ideal$--admissible ribbon graph on  $H$. 
   In what follows we extend the
colorings of coupons multilinearly.  In particular, this allows us to color an
ordered matching pair of coupons with the co-pairing $ \sum_i x_i\otimes_\kk y_i \in\Hom_\cat(\unit,P) \otimes_\kk \Hom_\cat(P, \unit ) $, which does not depend on the choice of basis:
\begin{equation}\label{E:OmegaDefMuilt}
\epsh{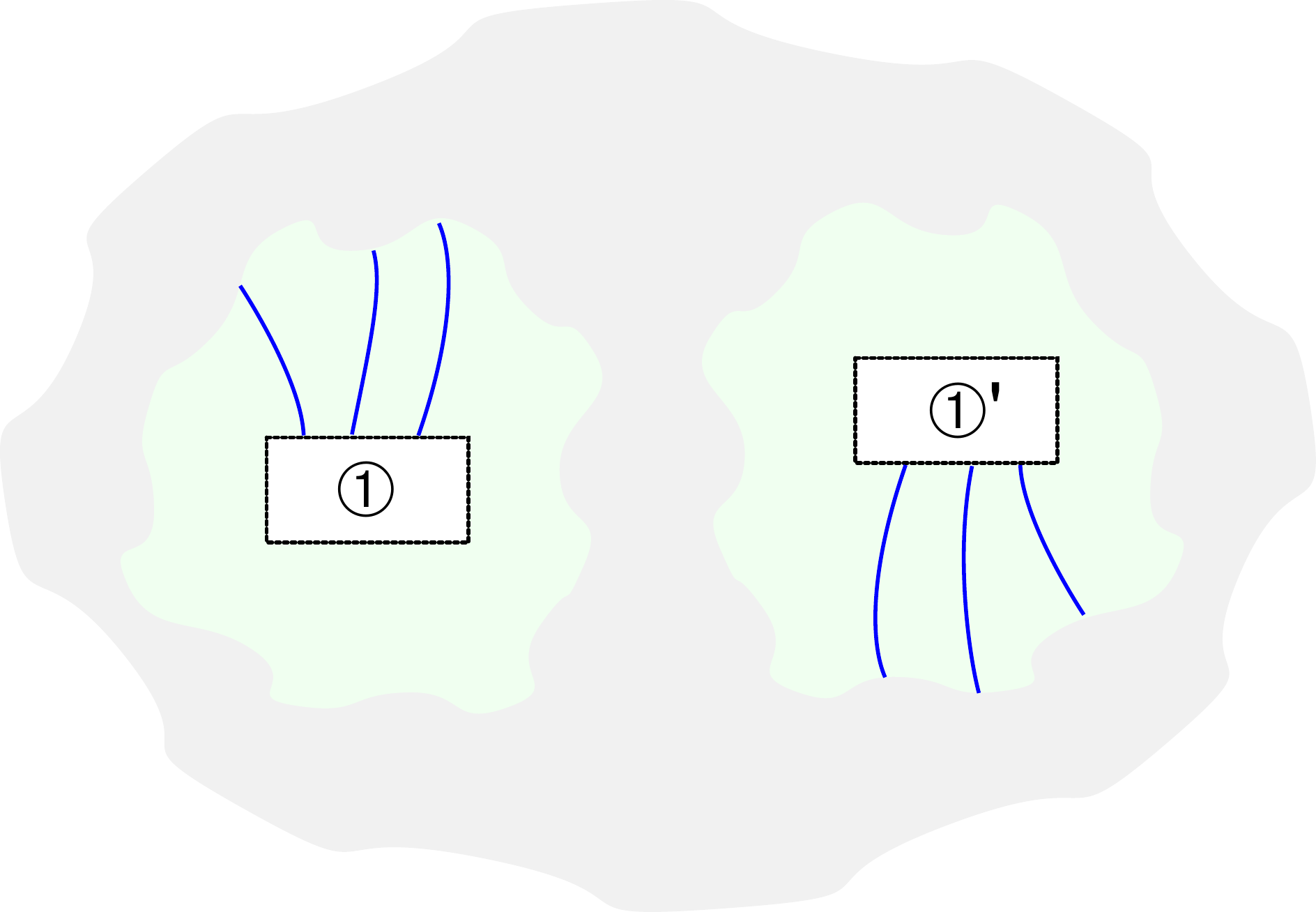}{20ex}=\sum_i\epsh{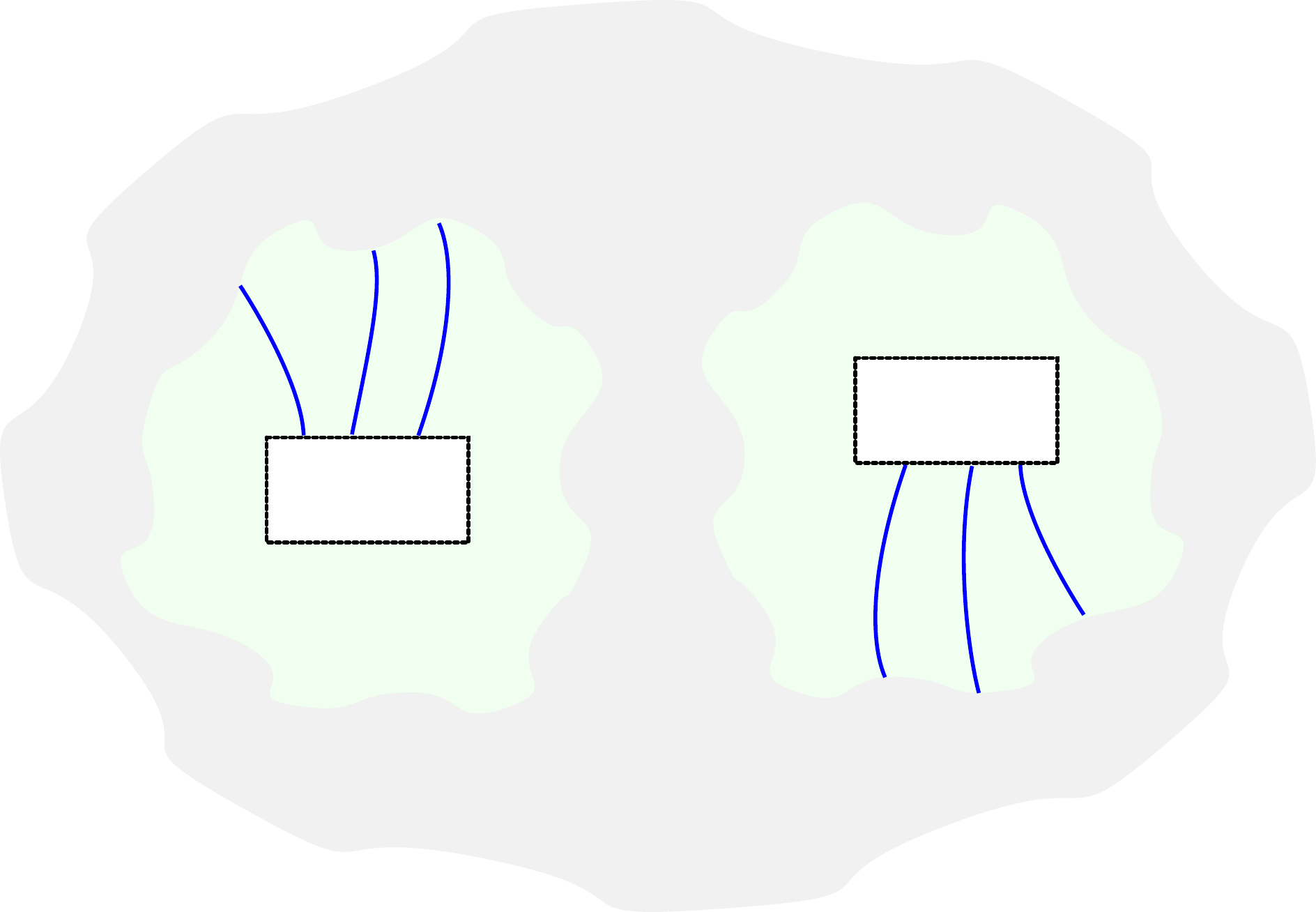}{20ex}\put(-95,-2){\ms{x_i}}\put(-39,6){\ms{y_i}}\;.
\end{equation}

We now  define a cutting operation  on colored graphs.
Let $(H,\Gamma)\in \Hp$ and let $D\subset H$ be an oriented properly embedded
disk  whose boundary
$\partial D \subset \partial H$ does not  meet  the coupons of $\Gamma$ and  intersects the edges of $\Gamma$
  transversely in a  non-empty set which contains at least one edge colored with an object $V$ which is in $\ideal$.  
Cutting $(H,\Gamma)$ along~$D$ we obtain  a new multi-handlebody graph
$(\cut_D (H),\cut_D(\Gamma)) \in \Hp$ where $\cut_D(\Gamma)$ is obtained by  cutting the edges of
$\Gamma$ intersecting $\partial D$  and then joining the cut points into two new coupons in $\partial (\cut_D (H))$  (one  on each side of the
cut). The coupons are  colored  as in Formula \eqref{E:OmegaDefMuilt}, see the following figure:
\begin{equation*}\label{E:CuttingPropF'}
  \epsh{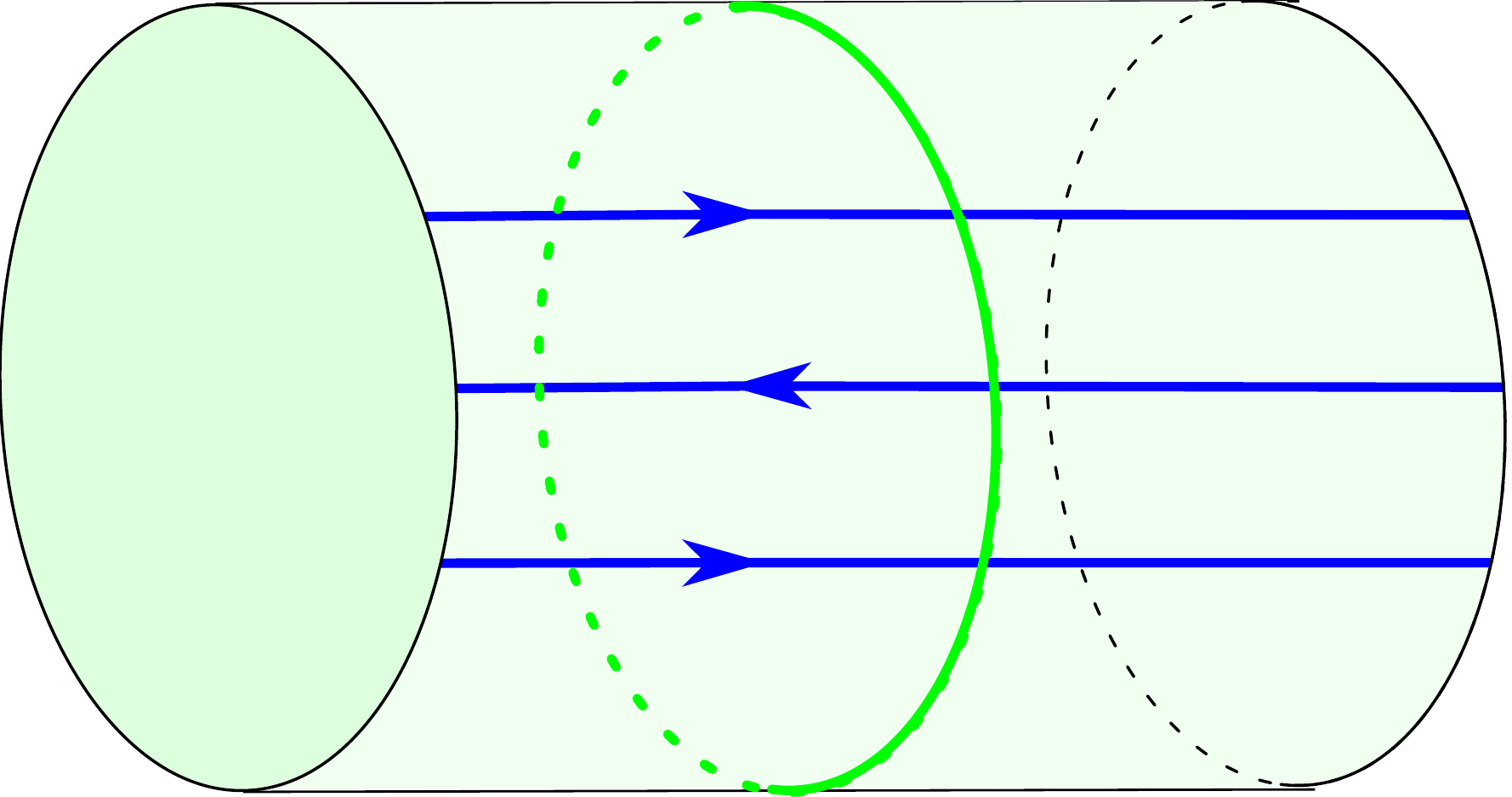}{9ex}\hspace{5pt} \longrightarrow \hspace{5pt} \epsh{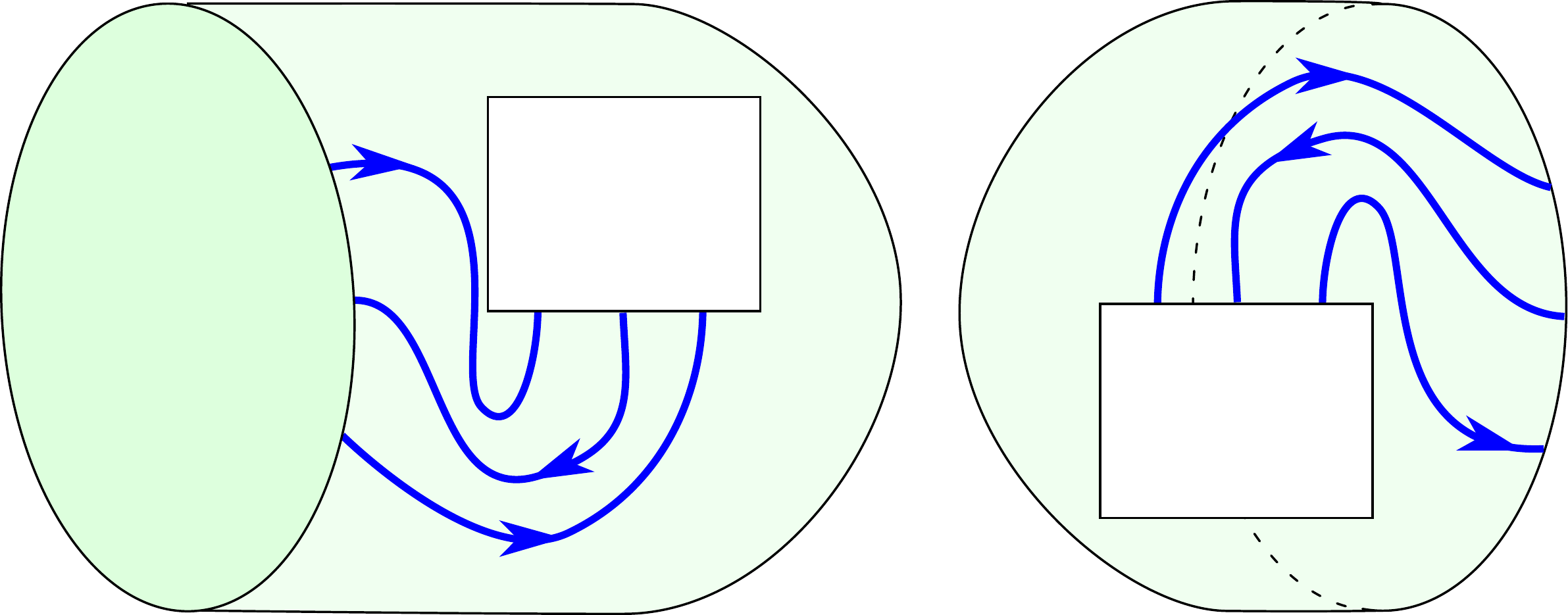}{9ex}
  \put(-65,6){\mbox{\textcircled{\ms{1'}}}}
    \put(-26,-6.5){\mbox{\textcircled{\ms{ 1}}}}\,.
\end{equation*}
Note that since $D$ intersects at least one $\ideal$-colored edge of $\Gamma$ then $\cut_D (H)$ is an element of $\Hp$. Moreover, $H$ and $\cut_D (H)$ can have different numbers of connected components and the orientation of $\cut_D (H)$ is induced by the one of $H$. The cutting of $(H, \Gamma)$ at~$D$ is not determined by these data uniquely as it  depends on the order in the set $\Gamma \cap \partial D$ compatible with the cyclic order in the oriented circle~$\partial D$.

The following theorem is a straight forward generalization of \cite[Theorem 2.1]{CGPT}  to $\ideal$--admissible ribbon graphs (in \cite[Theorem 2.1]{CGPT} all colors of a graph are required to be in $\ideal$).  
\begin{theorem}\label{P:F'onProjHandlebodies}
Let $\cat$ be a pivotal $\FK$-category equipped with an ideal $\ideal$ in $\cat$  and a non-degenerate m-trace $\mt$ on   $\ideal$.
Then there exists a unique mapping $F'_\mt:\Hp\to \FK$ 
satisfying the following four conditions.
\begin{enumerate}
\item 
$F'_\mt(H,\Gamma)\in \FK$  depends only on the orientation preserving diffeomorphism class of  $(H,\Gamma)\in \Hp$.
 
\item For any $\ideal$--admissible ribbon graph $(B^3,\Gamma)$ 
on the 3-ball $B^3$, we have
$
F'_\mt(B^3,\Gamma)=F'_{\mt}(\Gamma)
$ where $F'_{\mt}(\Gamma) $ is the value of the linear functional on $\Skein_\ideal(S^2)$ corresponding to the m-trace $\mt$, via Theorem~\ref{T:DiskRmt}.

\item For any $(H_1,\Gamma_1), (H_2,\Gamma_2)\in \Hp$
we have $F'_\mt(H_1\sqcup H_2,\Gamma_1 \sqcup \Gamma_2)=F'_\mt(H_1,\Gamma_1) F'_\mt(H_2,\Gamma_2).$

\item Cutting any $(H,\Gamma) \in \Hp$ along a 2-disk~$D$ as above, we have 
  $F'_\mt(\cut_D(H),\cut_D(\Gamma))=F'_\mt(H,\Gamma).$
\end{enumerate}
\end{theorem}
\begin{proof}
The proof follows by induction on the genus of $H$ and after a few very small replacements follows word for word as in \cite{CGPT}.  In particular, the induction base concerns graphs on disjoint unions of 3-balls.  Conditions (1) and (2) in Theorem~\ref{P:F'onProjHandlebodies} define $F'_\mt$ uniquely for $\ideal$--admissible graphs   on the boundary of a topological 3-ball.  Condition (3) extends 
  $F'_\mt$ uniquely to $\ideal$--admissible graphs on disjoint unions of 3-balls.  
  The proof that Condition (4) holds for an 
  $\ideal$--admissible graph $(B^3,\Gamma)\in \Hp$ on the boundary of a 3-ball $B^3$ follows word for word (after only considering disks $D$ that intersect at least one $\ideal$-colored edge of $\Gamma$ and replacing the words ``$\ideal$--colored graph'' with ``$\ideal$--admissible graph'') as in the proof of \cite[Theorem 2.1]{CGPT}.  Moreover, after a similar replacement the induction step follows word for word as in the proof of \cite[Theorem 2.1]{CGPT}.  (Note that up to an isotopy of $\Gamma$ one can always assume that any cutting disk intersects and $\ideal$--colored edge.)
\end{proof}

Next we state a purification process which produces a non-degenerate m-trace from a m-trace:
\begin{theorem}\label{T:purSimplifiedVersion}
  Let $\mt$ be a m-trace on an ideal $\ideal$.  Let  $\pi_\mt : \cat \to \cat/\wa\Nideal_{\mt}$ be the canonical
  monoidal functor (see the definition of a tensor ideal of morphisms and Lemma \ref{L:QuCatIdealMor}, below) where
  $$\wa\Nideal_{\mt} =\{f \in \Hom_\cat(V, W) \ | \  \mt_U(g_1\circ
  f\circ g_2)=0, \; \forall
  U\in\ideal,\forall g_1:W\to U,\forall g_2:U\to V\}_{ V,W\in \cat}$$
  is the kernel  of $\mt$.
  Let $\Jideal$ be the ideal of $\cat/\wa\Nideal_{\mt}$ generated by
  $\pi_\mt(\ideal)$. Then there is a unique non-degenerate m-trace
  $\wb\mt$ on $\Jideal$ such that $\mt=\wb\mt\circ\pi_\mt$.
\end{theorem}

We prove Theorem \ref{T:purSimplifiedVersion} in Subsection \ref{SS:TensorIdealsMor}.   Note from Proposition \ref{P:FunctorIndSkein} the functor $\pi_\mt$ induces a map $(\pi_\mt)_*$ on the level of skein modules.

\begin{theorem}\label{T:ActionOfHandelbody} Let $\cat$ be an essentially small pivotal $\FK$--category and let
  $\ideal$ be an ideal of $\cat$.  Let $\Sigma$ be a
  surface and $\mathcal{H}_\Sigma$ be the class of multi-handlebodies
  $H$ with $\partial H=\Sigma$.  Then there exist a map
  $$\Psi:\{\text{m-traces on } \ideal\} \times \mathcal{H}_\Sigma \to \Skein_{\ideal}(\Sigma)^* \text{ given by } \Psi(\mt,H) = \Big(\Gamma \mapsto F'_{\wb\mt} \big(H,(\pi_\mt)_*(\Gamma)\big) \Big) $$
 where $\pi_\mt$ and $\wb\mt$ are given in Theorem \ref{T:purSimplifiedVersion} and 
 $F'_{\wb\mt}$ is the invariant of Theorem \ref{P:F'onProjHandlebodies}.
 \end{theorem}  
\begin{proof}
We need to prove that if $\Gamma$ and $\Gamma'$ are $\ideal$--skein equivalent $\ideal$--admissible ribbon graphs in $\Sigma$ then
  $F'_{\wb\mt} (H,(\pi_\mt)_*(\Gamma))=F'_{\wb\mt} (H,(\pi_\mt)_*(\Gamma'))$.  The $\ideal$--skein equivalence between $\Gamma$ and $\Gamma'$ induces a $\Jideal$--skein equivalence between $(\pi_\mt)_*(\Gamma)$ and $(\pi_\mt)_*(\Gamma')$.  It is enough to assume this equivalence is given by a single skein relation supported in a
  box $B$.   Using Theorem \ref{P:F'onProjHandlebodies}, to compute both $F'_{\wb\mt} (H,(\pi_\mt)_*(\Gamma))$ and $F'_{\wb\mt} (H,(\pi_\mt)_*(\Gamma'))$ we can choose disjoint 2-disks $D_1,\dots,D_n$ which cut $H$ into a disjoint union of 3-balls $B^3_1,\dots B^3_m$.  Let $(\cut(H),\cut(\Gamma))$ and (resp. $(\cut(H),\cut(\Gamma'))$) be the graph on a multi-handlebody constructed from cutting $(H,(\pi_\mt)_*(\Gamma))$ (resp. $(H,(\pi_\mt)_*(\Gamma'))$) along the disks $D_1,\dots,D_n$ (here when cutting both $(H,(\pi_\mt)_*(\Gamma))$ and $(H,(\pi_\mt)_*(\Gamma'))$ we use the same basis and dual basis to fill the coupons).  Up to isotopy we can assume that the box $B$ does not intersect any of the boundaries of the disks $D_1,\dots,D_n$.   Then the box $B$ is contained in the boundary of exactly one of the balls $B^3_1,\dots B^3_m$, up to reordering say it is contained in $B^3_1$.  Then since $(\pi_\mt)_*(\Gamma)$ and  $(\pi_\mt)_*(\Gamma')$ are equal outside the box $B$ we have  $\cut(\Gamma) \cap B^3_i=\cut(\Gamma' )\cap B^3_i$ for all $i\neq 1$. Moreover, $\cut(\Gamma) \cap B^3_1$ and $\cut(\Gamma' )\cap B^3_1$ are $\Jideal$--skein equivalent, via the skein relation in $B$ and so Theorem \ref{T:DiskRmt} implies $F'_{\wb\mt}(\cut(\Gamma) \cap B^3_1)=F'_{\wb\mt}(\cut(\Gamma' )\cap B^3_1)$.  Thus, $F'_{\wb\mt}(\cut(H),\cut(\Gamma))=F'_{\wb\mt}(\cut(H),\cut(\Gamma'))$ completing the proof.  
  \end{proof}

\subsection{Purification and tensor ideals of morphisms}\label{SS:TensorIdealsMor}
Here we prove Theorem \ref{T:purSimplifiedVersion} by constructing a purification process  which is modeled
on usual semi-simplification construction.  
We believe this process will be useful outside this paper; in particular, in constructing relative modular categories, see \cite{AGP}.

Let $\cat$ be a pivotal $\FK$-category.
  A \emph{tensor ideal of morphisms} $\wa\Nideal$ in $\cat$ is a family of $\FK$--submodules
  $\wa\Nideal(V,W)\subset \Hom_\cat(V,W)$ which is closed under
  composition and tensor products with any morphisms in $\cat$,  where $V$ and $W$ run over all objects in $\cat$.
Tensor ideals are studied in \cite{HW2022} but our points of view are somewhat orthogonal:  \cite{HW2022} considers filtrations 
by an ideal of the ground ring whereas we work with linear categories
over a field $\FK$.
It would be interesting to understand how the two points of view are related.

\begin{lemma}
 The family $\wa\Nideal_{\mt}$ defined in Theorem~\ref{T:purSimplifiedVersion} is a tensor ideal of morphisms.
\end{lemma}
\begin{proof}
  Let $f:V\to W$ be in $ \wa\Nideal_{\mt}$ and let $h_1:W\to W'$ and
  $h_2:V'\to V$ be morphisms in $\cat$.  For any $U\in\ideal$, if
  $g_1:U\to V'$ and $g_2:W'\to U$ are morphism then
  $\mt_U(g_1( h_1fh_2 )g_2)=\mt_U((g_1h_1)f(h_2 g_2))=0$.  So
  $h_1fh_2\in \wa\Nideal_{\mt}$.
Given morphisms $g_1:U\to V'\otimes V\otimes V''$ and
  $g_2:V'\otimes W\otimes V''\to U$  let
  $$g'_1=(\lev_{V'}\otimes \Id_V \otimes \rev_{V''})(\Id_{(V')^*}\otimes g_1\otimes \Id_{(V'')^*}), \;\;
  g'_2=(\Id_{(V')^*}\otimes g_2\otimes \Id_{(V'')^*})(\rcoev_{V'}\otimes \Id_W\otimes \lcoev_{V''})
  $$ be morphism $g'_1: U'\to V$ and $g'_2:W\to U'$ where $U'=(V')^*\otimes U\otimes (V'')^*\in\ideal$. 
 Then
  $$\mt_{U}(g_2(\Id_{V'}\otimes f\otimes
  \Id_{V''} )g_1)=\mt_{U}\big(\ptr^{(V')^*}_L\big(\ptr^{(V'')^*}_R(g'_1fg'_2)\big)\big)=\mt_{U'}(g'_1fg'_2)=0$$
  thus $\Id_{V'}\otimes f\otimes \Id_{V''}\in\wa\Nideal_{\mt}$.  Combining this with the first part of the proof implies that $\Nideal_{\mt}$ is closed tensor product with any morphisms. 
\end{proof}
\begin{lemma}\label{L:QuCatIdealMor}
  If $\wa\Nideal$ is a tensor ideal of
  morphism in a pivotal $\kk$--category $\cat$, then the category  $\cat/\wa\Nideal$ whose
  objects are those of $\cat$ and morphisms are
  $\Hom(V,W)/\wa\Nideal(V,W)$ is a
  pivotal $\kk$--category.
\end{lemma}
\begin{proof}
  Since $\wa\Nideal$ is an ideal of morphisms, the composition and
  tensor product of morphisms in $\cat/\wa\Nideal$ is well defined and
  $\kk$-bilinear.  Then all identities in $\cat$ between composition
  and tensor of structure morphisms translate directly into the corresponding
  identities in $\cat/\wa\Nideal$.
\end{proof}

\begin{proof}[Proof of Theorem \ref{T:purSimplifiedVersion}]
  Objects of $\Jideal$ are retracts of elements
  $U'\otimes\pi(V)\otimes W'$ where $V\in\ideal$, and
  $U',W'\in\cat/\wa\Nideal$.  Since $U',W'$ are the image by $\pi$ of
  some $U,W\in\cat$ and $U\otimes V\otimes W\in\ideal$, then
  $U'\otimes\pi(V)\otimes W'\in\pi(\ideal)$.  Hence objects of
  $\Jideal$ are the objects $\pi(W)$ which are retract of $\pi(V)$ for
  some $V\in\ideal$.  Then there exists $i:W\to V$, $p:V\to W$ such
  that $pi=\Id_W+n$ in $\cat$ for some $n\in\wa\Nideal_{\mt}(W,W)$.
  For $\wb f\in\End_{\cat/\wa\Nideal_{\mt}}(\pi_{\mt}(W))$ we will
  show the assignment $\wb\mt_{\pi(W)}(\wb f)= \mt_V(ifp)$ defines a
  well defined unique m-trace on $\Jideal$.  If $\wb\mt$ exists then
  $\mt_V(ifp)=\wb\mt_{\pi(V)}(\pi(i)\wb f\pi(p))=
  \wb\mt_{\pi(W)}(\pi(p)\pi(i)\wb f)=\wb\mt_{\pi(W)}(\wb f)$ and so it
  is uniquely determined by $\mt$. To see it is well defined, let
  $f',i',p'$ be alternative choices then
  $f-f'\in\wa\Nideal_{\mt}(V,V)$ and
  $p'i'= pi=\Id_W \text{ mod } \wa\Nideal_{\mt}(W,W)$. Thus,
  $$\mt_V(ifp)=\mt_V(if'p)=\mt_V(if'p'i'p)=\mt_V(i'pif'p')=\mt_V(i'f'p').$$
 
  Finally, let us consider the non degeneracy of $\wb\mt$: 
  let $\pi_\mt(W)\in\Jideal$ (a retract of $\pi_\mt(V)$ with $V\in \ideal$) and let $\wb f: \pi_\mt(W)\to \pi_\mt(\unit)$ be a non-zero morphism in $\cat/\wa\Nideal_\mt$.  For any lift $f:W\to \unit$ of $\wb f$ is not contained in $\wa\Nideal_\mt$ and so there exists $U\in \ideal$ and morphisms $g_1:\unit\to U$ and $g_2: U\to W$ in $\cat$ such that $\mt_U(g_1fg_2)\neq 0$.  By definition $\wb\mt_{\pi(W)}((\wb g_2 \wb g_1) \wb f)=\mt_V(ig_2g_1fp)=\mt_U(g_1fg_2)\neq 0$.  Thus, $\wb f$ is non-zero in the pairing.  Analogously,  the pairing is non-degenerate on the left.  
  \end{proof}

To end the section we record some facts about the relationship between ideals and tensor ideals of morphisms.  These facts can be proven with a little work, we do not use these here. 
\begin{enumerate}
\item  If $\ideal\subset \cat$ is an additive  ideal (meaning any two objects of $\ideal$ have a direct sum in $\ideal$) 
  then the set of morphisms which factor through some object of
  $\ideal$ is a tensor ideal of morphisms $\Ar(\ideal)$.
  \item   If $\wa\Nideal$ is a tensor ideal of morphisms in $\cat$ then the
  set of $\wa\Nideal$-negligible objects (i.e.\ those objects
  $V\in \cat$ such that $\End_\cat(V)\subset \wa\Nideal$ or, equivalently
  $\Id_{V}\in \wa\Nideal$) is an ideal in $\cat$.
  \item  If $\ideal\subset \cat$ is an additive ideal then $\ideal$ is the ideal of
  $\Ar(\ideal)$-negligible object. 
  \item   If $\cat$ is idempotent complete (meaning any idempotent factors through a retract) and
  $\wa\Nideal$ is a tensor ideal of morphism in $\cat$ then
  the isomorphism classes of indecomposable objects of $\cat/\wa\Nideal$
  are in bijection with non $\wa\Nideal$-negligible isomorphism
  classes of indecomposable objects of $\cat$.
  \end{enumerate}

\section{Graded admissible skein modules}

\subsection{Homogeneous $\ideal$-admissible graphs} Let $\cat$ be a pivotal $\FK$--category and $\Gr$ be a
group.
A \textit{$\Gr$-grading} on $\cat$ is an equivalence of linear categories
$\cat \cong \bigoplus_{g \in \Gr} \cat_g$ where 
$\{ \cat_g \mid g \in \Gr \}$ is a  family of full subcategories of $\cat$
satisfying the following conditions: 1) $\unit \in \cat_1$, 2)  if $V\in\cat_g$,  then  $V^{*}\in\cat_{g^{-1}}$, 3) if $V\in\cat_g$, $V'\in\cat_{g'}$ then $V\otimes
    V'\in\cat_{gg'}$ and 4)  if $V\in\cat_g$, $V'\in\cat_{g'}$ and $\Hom_\cat(V,V')\neq \{0\}$, then
    $g=g'$.

\begin{definition}\label{def:homogeneous}
A $\cat$-colored graph $\Gamma$ is \emph{homogeneous} if each edge $e$ of $\Gamma$ belongs to $\cat_{g(e)}$ for some $g(e)\in \Gr$ and every coupon $c$ of $\Gamma$ is filled with a morphism of $\cat_{g(c)}$ for some $g(c)\in \Gr$.
\end{definition}
\begin{lemma}\label{lem:homogeneous}
Each $\ideal$--admissible graph is $\ideal$-skein equivalent to a linear combination of homogeneous $\ideal$--admissible graphs. 
\end{lemma}
\begin{proof}
  If $V\in \cat$  then the definition of a $\Gr$-grading 
  implies there exist finitely many homogeneous objects $V_i$
  such that $V=\oplus V_i$. Furthermore, if $V\in \ideal$ then each
  $V_i$ is in $\ideal$ because $V_i$ is a retract of $V$.  Let
  $p_i:V\to V_i$ and $j_i:V_i\to V$ be projections and inclusions such
  that $\Id_V=\sum_{i} j_i\circ p_i$.
 
  If $\Gamma$ contains an edge colored by $V$ then up to applying an
  $\ideal$-skein relation we can replace a small sub-arc of the edge by
  the formal sum of arcs each containing two coupons decorated
  respectively with $j_i$ and $p_i$ because
  $\Id_V=\sum_{i} j_i\circ p_i.$
  Then isotoping each of the coupons decorated by $j_i$ and $p_i$ we
  can ``absorb"
  them in the coupons at the
  endpoint of the edge via a $\ideal$-skein relation.
  Repeating this argument for all the non
  homogeneous edges we conclude.
\end{proof}

\subsection{Homology with $\Gr$-coefficients as a skein module}

Let $\Gr$ be a group. Define $\cat_\Gr$ to be the
$\FK$-category whose objects are $\{g, g\in \Gr\}$, endomorphisms are
$1$-dimensional and other morphisms are zero.  A strict  monoidal structure on $\cat_\Gr$ is given by $g\otimes h=gh$ and the pivotal
structure by $g^*=g^{-1}$ and the fact that all (co)-evaluation
morphisms are the identity of $\unit$.
The category $\cat_\Gr$ is
$\Gr$-graded with an unique object in each degree. If $\Gr$ is
abelian we endow $\cat_\Gr$ with the symmetric braiding 
$g\otimes h\to h\otimes g$ given by the identity map of $gh$ turning
$\cat_\Gr$ into a ribbon category.  
With this in mind, we let $M$ be a 2 or 3-manifold and in the second case we assume $\Gr$ is abelian so that $ \cat_\Gr$ is ribbon.

We will say that a $\cat_\Gr$-colored ribbon graph is \emph{group-like} if it is homogeneous and all
its coupons are colored by the identity morphism.  One can prove that any homogeneous graph is equivalent to $\lambda$ times a group-like
  ribbon graph for some $\lambda\in\FK^*$.  Furthermore, if $\Gamma$ is
  a group-like $\cat_\Gr$-colored ribbon graph in a box then its image
  by the Reshetikhin-Turaev functor $F$ is an identity morphism of $\cat_\Gr$.
 Also, one can check that $\Skein_{\cat_\Gr}(M)$ is the free $\FK$-vector space with
  basis the equivalence classes of group-like $\cat_\Gr$-colored ribbon graphs under the relation generated by $\Gamma\sim\Gamma'$ if $\Gamma-\Gamma'$ is a skein relation.
Using these facts one can prove the following lemma.  
\begin{lemma}\label{L:H1ab}
  If $\Gr$ is abelian then
  $\Skein_{\cat_\Gr}(M)=\FK[H_1(M;\Gr)]$ where the
  isomorphism is induced by sending a $g$-colored skein to the
  homology class of the $g$-colored element.  In particular the set of
  class of group-like element of $\Skein_{\cat_\Gr}(M)$ can be
  identified with $H_1(M;\Gr)$.
\end{lemma}

The above lemma motivates us to call the set of equivalence classes of group-like
elements of $\Skein_{\cat_\Gr}(M)$ the ``non commutative $H_1$'' of
$M$ with coefficients in $\Gr$ in general.  Remark that in
contrast with the commutative case this $H_1$ with non commutative
coefficients is not a group.
\begin{definition}
 We denote by $H_1(M;\Gr)$ the set of
  group-like elements of $\Skein_{\cat_\Gr}(M)$.
\end{definition}

\subsection{Graded admissible skein modules} Let $M$ and $\Gr$ be as in the last subsection.

\begin{definition}\label{D:holonomy}
  Let $\Gamma$ be a homogeneous admissible $\ideal$-colored graph in $M$. We say the \emph{degree of $\Gamma$} is the element
  $h_\Gamma\in H_1(M,G)$ represented by the group-like
  $\cat_G$-colored graph obtained by replacing all colors of edges of $\Gamma$ with their degrees (and each coupon with an identity morphism).
\end{definition}

\begin{proposition}\label{P:graduation}
  We have the following decomposition of $\FK$-vectors spaces:
 $\Skein_{\ideal}(M)=  \bigoplus_{h\in H_1(M;\Gr)} \Skein^h_{\ideal}(M)$
  where
  $\Skein^h_{\ideal}(M)$ is the $\FK$-span of homogeneous
  $\ideal$--admissible graphs $M$ of degree $h$.
\end{proposition}

\begin{proof}
  By Lemma \ref{lem:homogeneous} the natural $\FK$-linear map
  $\omega:\bigoplus_{h\in H_1(M;\Gr)} \Skein^h_{\ideal}(M)\to
  \Skein_{\ideal}(M)$ is surjective and the set of homogeneous
  $\ideal$--admissible graphs is a system of generators for
  $\Skein_\ideal(M)$.  We now claim that any skein relation in
  $\Skein_{\ideal}(M)$ is a sum of
  skein relations
  where each skein relation involves graphs of
  the same homogeneous degree; this will prove that $\omega$ is injective.
   Let
  $\sum_{i=1}^k c_i \Gamma_i=0\in \Skein_{\ideal}(M)$ be a skein
  relation associated to a box $B$. Applying the process of the proof
  of Lemma \ref{lem:homogeneous} we can split the $\Gamma_i$ as linear
  combination of graphs which are homogeneous outside $B$.  This
  reduces to the case of a skein relation where all $\Gamma_i$ are
  homogeneous outside $B$.  Now if $\Gamma$ and $\Gamma'$ are any two
  homogeneous graph which coincide outside the box $B$, we have
  $h_\Gamma=h_{\Gamma'}$ because they are both equivalent to the
  $\cat_G$-graph where the box $B$ is replaced with a coupon colored
  with the identity morphism.  Hence applying the process of the proof
  of Lemma \ref{lem:homogeneous} inside the box gives a linear
  combination of homogeneous graphs which have all the same degree.
\end{proof}

\subsection{Finite dimensionality}

In this subsection $\cat$ is an additive pivotal $\Gr$-graded $\FK$--category (the case of ungraded categories is included by setting $\Gr$ be the trivial group). 
Let $\Sigma$ be a surface.

\begin{definition}
  We say that an ideal $\ideal$ of $\cat$ is \emph{graded finite} if for each $g\in \Gr$
  the subcategory $\ideal\cap \cat_g$ admits a generator $V_g$, that is any object of $\ideal\cap \cat_g$ is retract of $V_g^{\oplus n}$ for some $n$.
\end{definition}
Recall that by Proposition \ref{P:graduation} we have
$\Skein_\ideal(\Sigma)=\bigoplus_{h\in H_1(\Sigma;\Gr)}
\Skein_\ideal^h(\Sigma)$.
\begin{proposition}\label{prop:finitewithboundary}
Let $\Sigma$ be a compact connected 
  surface with non-empty boundary then for every graded finite ideal
  $\ideal\subset \cat$ and every $h\in H_1(\Sigma;\Gr)$ it holds
  $\dim_{\FK}\Skein^h_\ideal(\Sigma)<\infty$. 
\end{proposition}
\begin{proof}
There exists a set of disjoint simple curves $c_1,\ldots c_k$ properly embedded in $\Sigma$ such that $\Sigma \setminus (c_1\sqcup \ldots \sqcup c_k)$ is a disk. 
If $\Gamma\in \Skein_\ideal(\Sigma)$ is homogeneous, then we can assume up to isotopy that it intersects each $c_i$ transversally and by Proposition \ref{prop:ddd} that all its edges are $\ideal$-colored. Then $\Gamma$ is skein equivalent (by fusing all the edges intersecting each $c_i$) to a graph intersecting each $c_i$ at most once; let $h_i\in \Gr$ be the degree of such an intersection and $V_{h_i}$ the generator of $\ideal\cap \cat_{h_i}$. Then up to applying one skein relation for each $c_i$ we can replace $\Gamma$ with a linear combination of graphs intersecting $c_i$ via a single edge colored by $V_{h_i}$.
We can express each of these graphs via linear combination of graphs each containing a single coupon decorated by a morphism in $\Hom(\unit,V_{h_1}\otimes \cdots \otimes V_{h_k})$ in the disk $\Sigma\setminus (c_1\sqcup\cdots \sqcup c_k)$. 
Therefore if $h\in H_1(\Sigma, \Gr)$ is the homology class of $\Gamma$ we just showed that $\dim_\FK \Skein_\ideal^h(\Sigma)\leq \dim_{\FK}\Hom(\unit,V_{h_1}\otimes \cdots \otimes V_{h_k})<\infty$.
   \end{proof}
   \begin{proposition}\label{P:G_I}
  Let $\ideal$ be an ideal in $\cat$, then the set
  $\Gr_\ideal=\{g\in\Gr:\ideal\subset \ideal_{\cat_g}\}$
  is a subgroup of
  $\Gr$ where $\ideal_{\cat_g}$ is the ideal generated by all object of $\cat_g$.
\end{proposition}
\begin{proof}
  Let $g,g'\in\Gr_\ideal$ and $V\in\ideal$, then $V$ is a retract of $U\otimes W'$ for some $U\in \cat$ and $W'\in\cat_{g'}$ implying $V$ is also a retract of
   $(V\otimes V^*\otimes U)\otimes W'$ (using $\lcoev_V$ and
  $\lev_V$).  So we can assume $U\in\ideal$ but then $U$ is retract of $U'\otimes W$ for some $U'\in \cat$ and $W\in\cat_g$ implying $V$ is a retract of $U'\otimes (W\otimes W')$. Thus,  $gg'\in\Gr_\ideal$.  A similar argument for the left
  tensor product implies that $g'g\in\Gr_\ideal$.
\end{proof}

   \begin{theorem}\label{T:finitewithoutboundary}
Let $\ideal$ be a graded finite ideal of $\cat$.  Assume $\Gr$ is abelian or $\Gr=\Gr_\ideal$.    Let $\Sigma$ be a compact oriented
  surface without boundary then for every $h\in H_1(\Sigma;\Gr)$ it holds
  $\dim_{\FK}\Skein^h_\ideal(\Sigma)<\infty$. 
  \end{theorem}

\begin{proof}
It is sufficient to prove the statement for $\Sigma$ a compact connected surface.  Consider a cellularization of $\Sigma$ consisting in a single vertex $v$, $2g$ closed curves $c_1,\ldots, c_{2g}$ and one disk $D$. If $\Gamma$ is a homogeneous graph of degree $h$, we can as above suppose that it intersects each $c_i$ transversally and that it is  contained in $\Sigma\setminus v$. Arguing as in the proof of Proposition \ref{prop:finitewithboundary} we see that $\Gamma$ is skein equivalent to a linear combination of graphs of the form of a bouquet of circle $B$ where each edge intersects a single $c_i$ once and is colored by some $V_{h_i}\in \ideal$ and these edges end up in a single coupon colored by some $f\in \Hom_{\cat}(\unit,V_{h_1}\otimes V_{h_2}\otimes V_{h_1}^*\otimes V_{h_2}^*\otimes \cdots\otimes V_{h_{2g-1}}\otimes V_{h_{2g}}\otimes V_{h_{2g-1}}^*\otimes V_{h_{2g}}^*)$ contained in $D$.  In particular, $h$ can be represented by $h_B$ which is  the bouquet $B$ where the edge intersecting $c_i$ is colored by $h_i$.  

We claim if $\Gamma'$ is a homogeneous graph of degree $h$ then there exists $\Gamma''$ which is skein equivalent to $\Gamma'$ such that for every $i$, $\Gamma''$ has a single edge interesting  $c_i$ which is colored by $V_{h_i}$.   This will conclude the proof because   
$$\dim_\FK(\Skein_\ideal^h(\Sigma))\leq \dim_\FK \Hom_{\cat}(\unit,V_{h_1}\otimes V_{h_2}\otimes V_{h_1}^*\otimes V_{h_2}^*\otimes \cdots\otimes V_{h_{2g-1}}\otimes V_{h_{2g}}\otimes V_{h_{2g-1}}^*\otimes V_{h_{2g}}^* )<\infty.$$

Let us prove the claim.  Recall the group-like
$\cat_G$-colored graph $h_\Gamma$ of $\Gamma$ given in Definition \ref{D:holonomy}.  Since $h_{\Gamma'}$ is skein equivalent to $h_{B}$ in $\Skein_{\cat_\Gr}(\Sigma)$ there exists a finite sequence of isotopies and $\cat_\Gr$--skein relations with boxes contained in the 2-cell $D$ transforming $h_{\Gamma'}$ into $h_{B}$.  A skein relation in $D$ does not change the degrees induced on the edges.  If during an isotopy $h_{\Gamma'}$ does not cross the vertex $v$ then the degrees induced on each edge $c_i$ are unchanged during the isotopy.  In the case when $\Gr$ is abelian the intersection  degree does not change even if isotopy crosses the vertex $v$.

\vspace{3pt}
\begin{minipage}{0.6\linewidth}
  Finally, in the case $\Gr=\Gr_\ideal$, if an isotopy has a
  $g\in \Gr$ colored edge crossing $v$ then assuming there is a
  $W_l\in \ideal \cap \cat_{l}$ colored edge near $v$, we apply the
  skein equivalence given in the following figure to $\Gamma'$ where
  the coupons in the graph represents $W_l$ as a retract of
  $W_g\otimes W_{g^{-1}l}$ which exist since $l\in \Gr_{\ideal}$.
\end{minipage}
\begin{minipage}{0.4\linewidth}
  \begin{equation*}
  \epsh{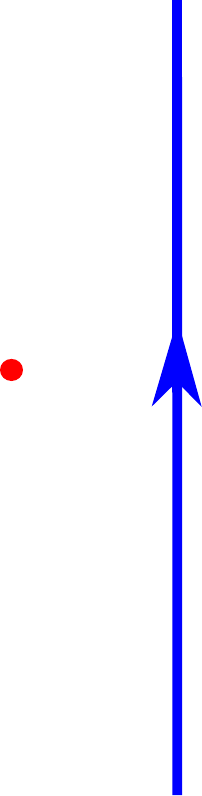}{14ex}\put(-20,5){\ms v}\put(0,5){\ms{W_l}}\quad\longrightarrow\epsh{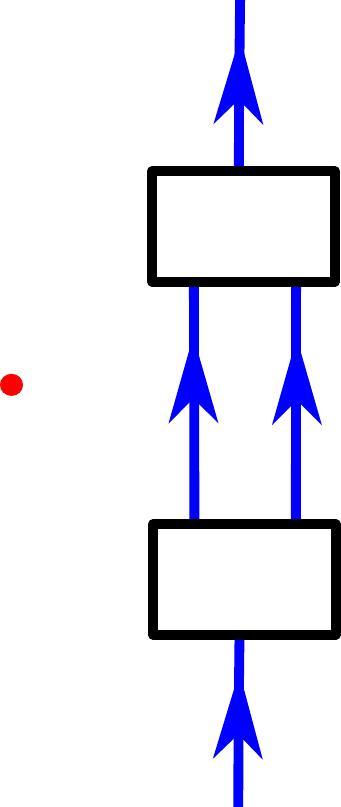}{14ex}\longrightarrow\epsh{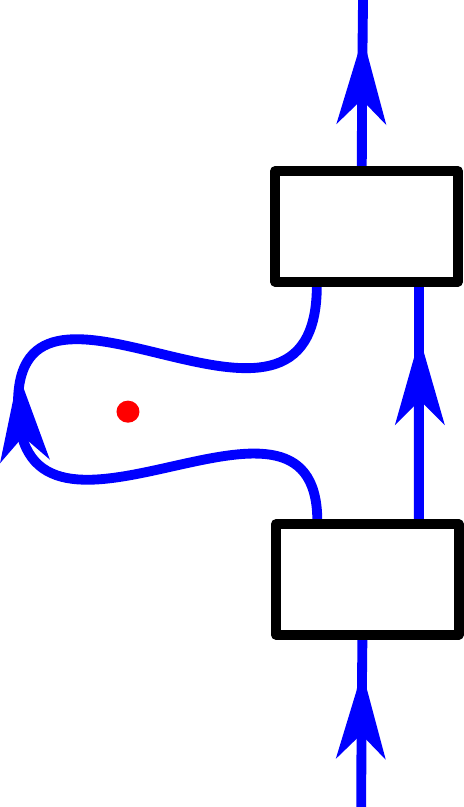}{14ex}
\end{equation*}
\end{minipage}
After these modifications on $\Gamma'$ we fuse the resulting graph as
we did for $\Gamma$ in the complement of $v$ then the resulting graph
is $\Gamma''$.
   \end{proof}

\begin{remark}
One can prove in a similar way that if $\cat_\Gr$ is ribbon and $\ideal$ is a graded finite ideal of $\cat$ then similar statement to Theorem \ref{T:finitewithoutboundary} holds for compact 3-manifolds.  
\end{remark}

\section{Examples}\label{S:Examples}
\subsection{Modular categories}
In this subsection we interpret skein modules as TQFT spaces for certain semisimple categories.
Suppose that $\cat$ is a modular category in the sense of Turaev (i.e.\ it is ribbon, finite semisimple, and with invertible $S$-matrix). Here, since $\cat$ is semisimple its only ideal is $\cat$ itself and the only m-trace (up to scalar) is the usual quantum trace.  In this case, the definition of an admissible skein module reduces to the usual definition.   Let $\ZZ_{RT,\cat}$ be the Reshetikhin--Turaev  TQFT functor associated to $\cat$ as in \cite{Tu}.  
The following lemma is due to J.\ Roberts, a proof can be found in \cite{Si2000}.
\begin{lemma}\label{lem:skeinboundary}
Given a compact $3$-manifold $M$ whose boundary $\Sigma$ is a possibly disconnected closed surface then $\ZZ_{RT,\cat}(\Sigma)=\Skein_\cat(M)$.
\end{lemma}
If the boundary of a compact oriented surface $\Sigma$ is non-empty then one can check that $\Skein_{\cat}(\Sigma)$ is isomorphic to $\ZZ_{RT,\cat}(D(\Sigma))$ where $D(\Sigma)$ is the handlebody   obtained by thickening $\Sigma$. 
\begin{lemma}\label{lem:endomorphisms}
  Let $\Sigma$ be a closed surface and $H$ be a handlebody
  with $\partial H=\Sigma$.  Then $\Skein_{\cat}(\Sigma)$ is
  isomorphic to $End(\ZZ_{RT,\cat}(\Sigma))$. 
  \end{lemma}
  
\begin{proof}
 Lemma \ref{lem:skeinboundary} implies
 $$\Skein_\cat(\Sigma)=\Skein_\cat(\Sigma\times [-1,1])=\ZZ_{RT,\cat}(\Sigma\sqcup \overline{\Sigma})=\ZZ_{RT,\cat}(\Sigma)\otimes \ZZ_{RT,\cat}(\overline{\Sigma})=\ZZ_{RT,\cat}(\Sigma)\otimes \ZZ_{RT,\cat}({\Sigma})^*=\operatorname{End}(V(\Sigma)).$$
\end{proof}

Finally, let $\catd$ be the spherical fusion category and let $\ZZ_{TV,\catd}$ be its associated  Tureav-Viro TQFT.
It follows from \cite{Bart2022} (see also \cite{Kirillov2012}) that if $\Sigma$ is an oriented closed surface then $\ZZ_{TV}(\Sigma)\simeq \Skein_\catd(\Sigma)$ and the representation of the mapping class group of $\Sigma$ coming from $\ZZ_{TV,\catd}$ is isomorphic to the representation given in Proposition~\ref{P:MappingClassActs}.

\subsection{Unrestricted quantum group}
Let $\ell\ge2$, $\xi=e^{2i\pi/\ell}$ and $r=\ell/2$ if $\ell$ is even and  $r=\ell$ otherwise.
Consider the unrestricted quantum group $U=U_q(\mathfrak{sl}_2)$ with generators $E,F,K^{\pm1}$ and relations $KE=q^2EK$, $FK=q^2KF$, $[E,F]=\frac{K-K^{-1}}{q-q^{-1}}$.
Let $\cat$  be the category of finite dimensional weight modules over $U$ and let $\Proj$ be its ideal of projective objects.    From \cite{GP2013Top},  $\cat$ is a pivotal $\C$--category and is graded by the group $\Gr$ of pairs $\bigl( \bigl( \begin{smallmatrix}1 &\ve \\ 0 & \kappa \end{smallmatrix}\bigr), \bigl( \begin{smallmatrix}\kappa &0 \\ \vp & 1 \end{smallmatrix}\bigr) \bigr)$ for $\ve,\vp\in\C,\kappa\in\C^*$,
where the group structure is given by component wise matrix multiplication.  
In particular, for each
$g\in\Gr$
let $I_g$ be the ideal of $U$ generated by $K^r-\kappa$, $E^r-(q-q^{-1})^{-r}\ve$ and $F^r-(-1)^\ell(q-q^{-1})^{-r}\vp\kappa^{-1}$, then $\cat_g$ is the full subcategory of $\cat$ whose objects are modules where $I_g$ acts by zero.
Moreover, $\cat_g$ is the category of finite dimensional 
$U_g$-modules where $U_g=U/I_g$ and its left regular representation is a generator of $\Proj$ in degree $g$; implying $\Proj$ is graded finite.
  Note that  $\Gr=\Gr_{\Proj}$ (see Proposition \ref{P:G_I})
  because
 $\Proj$ is contained in every ideal.
 Thus, by Theorem \ref{T:finitewithoutboundary} we have $\Skein^g_{\Proj}(\Sigma)$ is finite dimensional.  
The TV-type TQFT spaces of \cite{GP2013Top} naturally map to $\Skein_{\Proj}$ in a way which is compatible with the action of the mapping class group.

\subsection{Basic classical Lie superalgebras}
Here we discuss rich examples of nested ideals coming from a basic classical Lie superalgebra $\g=\g_{\p 0}\oplus \g_{\p 1}$, for more details see for example \cite{GKP1, GKP2}.   Let $\cat$ be the category of finite dimensional $\g$-supermodules which are completely reducible as $\g_{\p 0}$-modules and all $\g$-supermodule homomorphisms which preserve the $\Z_2$-grading.
For $V\in\cat$ let $\ideal_V$ be the ideal of all supermodules which appear as a direct summand of $V\otimes W$  for some $W$ in $\cat$.  Given a simple $\g$-supermodule $L(\lambda)$
 of highest weight $\lambda$, let $\atyp(\lambda)$ denote the atypicality of $\lambda$.  In \cite[Conjecture 6.3.2]{GKP1}, the following generalization of the Kac-Wakimoto conjecture was given:
\begin{conjecture} Let $\g$ be a basic classical Lie superalgebra. Then the simple $\g$-supermodule $L(\lambda)$ admits an m-trace $\mt$ on $\ideal_{L(\lambda)}$.  If $L(\mu)$ is another simple supermodule with
$\atyp(\mu) \leq \atyp(\lambda)$,
then $L(\mu)$ is an object of $\ideal_{L(\lambda)}$, and
$\atyp(\mu) \leq \atyp(\lambda)$ if and only if $\mt_{L(\mu)}(\Id_{L(\mu)})\neq 0.$
\end{conjecture}
This conjecture has been proved in several cases: Serganova \cite{serganova4} proved it for $\g=gl(m|n)$ and Kujawa \cite{Kuj2012} proved it for $\g=osp(m|2n)$.  For these $\g$ the conjecture implies there exists ideals $\ideal_k$ for $k=0,1,...,\operatorname{min}(m,n)$ one for each  atypicality such that $\Proj=\ideal_0 \subset \ideal_1 \subset ... \subset \ideal_{\operatorname{min}(m,n)}=\cat$
where each of these ideas has a non-zero m-trace.  Thus, Theorems \ref{T:DiskRmt} and \ref{T:dim3SphereSkein} imply that 
$\Skein_{\ideal_k}(B^3),$ $\Skein_{\ideal_k}(S^3),$ $\Skein_{\ideal_k}(D^2)$ and $\Skein_{\ideal_k}(S^2) $
are all non-trivial and one dimensional because $\ideal_k$ is generated by a simple module.

\linespread{1}

\end{document}